\documentclass[12pt]{amsart}

\usepackage{hyperref}
\hypersetup{nesting=true,debug=true,naturalnames=true}
\usepackage{graphicx,amssymb,upref}

\usepackage{amsmath,amsthm}
\usepackage{amsfonts}
\usepackage{latexsym}
\usepackage{amsbsy}
\usepackage{amssymb,mathscinet}

\numberwithin{equation}{section}

\usepackage{amsfonts,amssymb,amscd,amsmath,enumerate,verbatim,calc,enumerate}
\theoremstyle{plain}
\newtheorem{theorem}{Theorem}[section]
\newtheorem{corollary}[theorem]{Corollary}
\newtheorem{lemma}[theorem]{Lemma}
\newtheorem{proposition}[theorem]{Proposition}

\newtheorem{notation}[theorem]{Notation}
\newtheorem{definition}[theorem]{Definition}

\newtheorem{remark}[theorem]{Remark}
\newtheorem{example}[theorem]{Example}

\newcommand{\be}{\mathbb E}
\newcommand{\bn}{\mathbb N}
\newcommand{\Nk}{\bn_0^k}
\newcommand{\ot}{\otimes}
\newcommand {\id} {{\textrm{id}}}
\newcommand{\wt}{\widetilde}
\newcommand{\wT}{\wt{T}}
\newcommand{\onek}{\{1,\ldots,k\}}

\begin{document}

\title[]{Doubly commuting invariant subspaces for representations of product systems of $C^*$-correspondences}

\date{\today}


\author[Trivedi]{Harsh Trivedi\textsuperscript{*}}

\address{The LNM Institute of Information Technology, Rupa ki Nangal, Post-Sumel, Via-Jamdoli
	Jaipur-302031,
	(Rajasthan) INDIA}
\email{trivediharsh26@gmail.com, harsh.trivedi@lnmiit.ac.in}

\author[Veerabathiran]{Shankar Veerabathiran}
\address{Ramanujan Institute for Advanced Study in Mathematics,
	University of Madras, Chennai (Madras) 600005, India}
\email{shankarunom@gmail.com}
\thanks{*corresponding author}






\begin{abstract}
	We obtain a Shimorin-Wold-type decomposition for a doubly commuting covariant representation of a product system of $C^*$-correspondences. This extends a recent Wold-type decomposition by Jeu and Pinto \cite{MP} for a $q$-doubly commuting isometries. Application to the wandering subspaces of doubly commuting induced representations is explored, and a version of Mandrekar's Beurling type theorem is obtained to study doubly commuting invariant subspaces using Fock space approach due to Popescu.
\end{abstract}

\maketitle
\section{Introduction}

The theorem of Wold \cite{W}, known as Wold decomposition, says
that every isometry on a Hilbert space decomposes uniquely as a direct sum of a shift operator and a unitary operator. A classical  application of the Wold decomposition is Beurling's theorem \cite{B49} which states that if $\mathcal K$ is a closed invariant subspace for the shift $M_z$ of the Hardy space $H^2(\mathbb D),$ then $\mathcal K$ is the image of an inner function. A wandering subspace theorem due to Halmos \cite{H61} is a generalization of the Beurling's theorem. In \cite{R69}, Rudin explained that the Beurling theorem fails in general in the multivariable case, that is, for commuting tuple of isometries. S{\l}oci{\'n}ski
in \cite{Sl80} proved a Wold-type decomposition for doubly commuting isometries, which is based on a characterization of the existence of a wandering subspace for commuting tuple of isometries in terms of the doubly commuting condition. Mandrekar \cite{M88} utilized the S{\l}oci{\'n}ski's decomposition and gave a  Beurling's type theorem for the Hardy space over the bidisc $H^2(\mathbb D^2)$ under the assumption that the tuple of multiplication operators by co-ordinate functions $(M_{z_1}, M_{z_2})$ is doubly commuting. Sarkar, Sasane and Wick \cite{S13} developed the Mandrekar's result, in the polydisc case $H^2(\mathbb D^n),$ based on a multivariable version of the S{\l}oci{\'n}ski's result given in \cite{S14}. There is a recent book \cite{MR18} on this theme by Mandrekar and Redett.

Pimsner in \cite{P97} extended the notion of Cuntz-Kreiger algebras using the terminology of isometric covariant representations of $C^*$-correspondences. The $C^*$-representations of tensor algebras of a $C^*$-correspondence are in a bijective correspondence with the isometric covariant representations of the $C^*$-correspondence (cf. \cite{P97}).
In \cite{MS99}, Muhly and Solel derived an analogue of the Wold decomposition for the isometric covariant representations which is a generalization of the Wold decomposition for a row
isometry by Popescu \cite{Po89}. Arveson explored the notion of tensor product system of Hilbert spaces in \cite{A89} to classify $E_0$-semigroups. Solel in \cite{S08} introduced the notion of doubly commuting covariant representations of product systems of $C^*$-correspondences and explored theory of regular dilations.
Skalski and Zacharias generalized S{\l}oci{\'n}ski's Wold-type decomposition for the doubly commuting isometric covariant representations.

 Shimorin \cite{S01} generalized the Wold decomposition for an operator close to isometry. A version of Shimorin-Wold-type decomposition for covariant representations close to isometric is proved in \cite[Theorem 3.13]{HS19}. We obtain an analogue of this result for the doubly commuting covariant representations in Section \ref{SecW}, which provides a different proof of the result of Skalski and Zacharias, based on the techniques developed in \cite{S14,MP}. In the setting of $C^*$-correspondences induced representations, introduced by Rieffel \cite{R74}, plays the role of a shift.  In this section the wandering subspace theorem for commuting shift \cite[Theorem 3.3]{S14} is extended for a covariant  representation of a product system. Final section is devoted to the Beurling-Lax type characterization of the doubly commuting invariant subspaces.

\subsection{Preliminaries: Shimorin Wold-type decomposition of single covariant representation}

 In this subsection we recall few definitions and elementary properties of covariant representations of $C^*$-correspondences from (see
\cite{La95, MR1648483, MR0355613, P97}).

Let $E$ be a Hilbert
$C^*$-module over a $C^*$-algebra $\mathcal M.$ By $\mathcal L(E)$ we denote the $C^*$-algebra of all
adjointable operators on $E$. We say that the  module $E$ is a {\it
$C^*$-correspondence over $\mathcal M$} if there exists a left $\mathcal
M$-module structure through a non-zero $*$-homomorphism
$\phi:\mathcal M\to \mathcal L(E)$ in the following sense
\[
a\xi:=\phi(a)\xi \quad \quad (a\in\mathcal M, \xi\in E).
\]
Each $*$-homomorphism considered in this article is
essential, that means, the closed linear span of
$\phi(\mathcal M)E$ is $E.$ Every $C^*$-correspondence has usual operator space structure induced from viewing it as a corner in respective linking algebra. If $F$ is another
$C^*$-correspondence over $\mathcal M,$ then we may consider the notion of
tensor product $F\otimes_{\phi} E$ (cf. \cite{La95}) which satisfy
\[
(\zeta_1 a)\otimes \xi_1=\zeta_1\otimes \phi(a)\xi_1,
\]
\[
\langle\zeta_1\otimes\xi_1,\zeta_2\otimes\xi_2\rangle=\langle\xi_1,\phi(\langle\zeta_1,\zeta_2\rangle)\xi_2\rangle
\]
for every $\zeta_1,\zeta_2\in F;$ $\xi_1,\xi_2\in E$ and $a\in\mathcal
M.$

\begin{definition}
Let $\mathcal H$ be a Hilbert
space, and $E$ be a $C^*$-correspondence over a $C^*$-algebra $\mathcal M$. Let
$\sigma:\mathcal M\to B(\mathcal H)$ be a representation and $T:
E\to B(\mathcal H)$ be a linear map. Then the pair $(\sigma,T)$ is
said to be a {\rm covariant representation} of $E$ on $\mathcal H$ if
\[
T(a\xi a')=\sigma(a)T(\xi)\sigma(a') \quad \quad (\xi\in E;
a,a'\in\mathcal M).
\]
We say that the covariant representation is {\rm completely bounded} (respectively, {\rm completely
contractive}) if $T$ is  completely bounded (respectively, completely contractive). Moreover, it is called {\rm isometric} if
\[
T(\xi)^*T(\zeta)=\sigma(\langle \xi,\zeta\rangle) \quad \quad
(\xi,\zeta\in E).
\]
\end{definition}
\begin{example}
 Let $\mathcal{K}$ be a finite or countably infinite dimensional  Hilbert space and $\sigma$ be a representation of the complex number $\mathbb{C}$ on $\mathcal{K}.$ Let $(\sigma, T)$ be a completely contractive covariant representation of $E(:=\mathcal{K})$ on a Hilbert space $\mathcal{H}.$ If $\{e_n\}_{n=1}^{m}(m:=\mbox{dim }\: \mathcal{K})$  is an orthonormal basis for $E,$ then  $T(\xi)=\sum\langle \xi, e_n \rangle T_n$ where $T_n:=T(e_n),  1\leq n \leq m.$ Moreover, since $(\sigma, T)$ is completely contractive, $\sum_{n=1}^{m} T_nT^*_n \leq I.$ Conversely, let $ m \in \mathbb{N}\cup \{ {\infty}\}$ and  $\{T_n\}_{n=1}^{m}$ be a sequence of operators on the Hilbert space $\mathcal{H}$ such that $\sum_{n=1}^{m} T_nT^*_n \leq I.$
 Define a completely contractive covariant representation $(\sigma, T)$ of $E$  on $\mathcal{H}$ by $T(e_n):=T_n$ and $\sigma(x):=xI_{\mathcal{H}},$
 where $E$ is  Hilbert space whose dimension is $m,$ with the orthonormal basis $\{e_n\}_{n=1}^{m}.$
 Therefore, there is a one-to-one correspondence between the completely contractive covariant representation of $E$ on Hilbert space $\mathcal{H}$ and the sequence $\{T_n\}_{n=1}^{m}(m:=\mbox{dim }\: E)$ of operators on  $\mathcal{H}$ satisfying $\sum_{n=1}^{m} T_nT^*_n \leq I.$ For more details and connection with the dilation theory of contractions and the Wold decomposition, see \cite{MR1648483,MS99}.
\end{example}

\begin{lemma}(\cite[Lemma
	3.5]{MR1648483})
The mapping $(\sigma, T)\mapsto \widetilde T$ gives a bijection
between the collection of all completely bounded (respectively, completely contractive), covariant
representations $(\sigma, T)$ of $E$ on $\mathcal H$ and the
collection of all bounded (respectively, contractive) linear maps
$\widetilde{T}:~\mbox{$E\otimes_{\sigma} \mathcal H\to \mathcal
H$}$
 which satisfies  $\widetilde{T}(\phi(a)\otimes I_{\mathcal
H})=\sigma(a)\widetilde{T}$, $a\in\mathcal M$. Moreover, the covariant representation $(\sigma, T)$ is isometric if and only if  $\widetilde
T$ is an isometry.
\end{lemma}
We say that the covariant representation $(\sigma, T)$ is {\it fully co-isometric} if $\widetilde{T}$ is co-isometric, that is, $\widetilde{T}\widetilde{T}^*=I_{\mathcal{H}}.$ Let $E$ be a $C^*$-correspondence over a $C^*$-algebra $\mathcal M.$ Then for each $n\in \mathbb{N},$ let us recall the following notation and definition of Wandering subspaces from \cite{H16}:

\begin{definition}
Let $E$ be a $C^*$-correspondence over a $C^*$-algebra $\mathcal M.$ Let $(\sigma , T)$ be a completely bounded, covariant representation of $E$ on a Hilbert space $\mathcal H$. For a closed $\sigma(\mathcal{M})$-invariant subspace $\mathcal{W}$ of $\mathcal{H},$ we define $$\mathfrak{L}_{n}(\mathcal{W}):=\bigvee \{T(\xi_1)T(\xi_2)\cdots T(\xi_n)h \: : \: \xi_i\in E , h \in \mathcal{W} \},$$ for $n\in\mathbb N$ and $\mathfrak{L}_{0}(\mathcal{W}):=\mathcal{W}$. Then $\mathcal W$ is called {\rm wandering} for the covariant representation $(\sigma , T)$, if
 $\mathcal{W}$ is orthogonal to the subspaces $\mathfrak{L}_{n}(\mathcal{W}),$ for all $n \in \mathbb{N}_0.$ We say that $(\sigma, T)$ has {\rm generating wandering subspace property} if
 $$\mathcal{H}=\bigvee_{n \in \mathbb{N}_0}\mathfrak{L}_{n}(\mathcal{W}), $$ for some {\rm wandering} subspace $\mathcal{W}$ of $\mathcal{H},$ and we call $\mathcal{W}$ as {\rm generating wandering subspace} for $(\sigma, T).$
\end{definition}
Note that if $(\sigma, T)$ is isometric covariant representation of $E$ on $\mathcal{H},$ then $\mathcal{W}$ is wandering subspace for $(\sigma, T)$ if and only if  $$\mathfrak{L}_{n}(\mathcal{W})\: \bot \: \mathfrak{L}_{m}(\mathcal{W}),$$ for all   distinct $n,m \in \mathbb{N}_0.$
\begin{definition}
	Let  $\mathcal K$ be a closed subspace of $ \mathcal H.$  The subspace $\mathcal K$ is called $(\sigma, T)$-{\it invariant} (respectively, $(\sigma, T)$-{\it reducing}) (cf. \cite{SZ08}) if it   is $\sigma(\mathcal M)$-invariant (i.e. the projection onto $\mathcal K$, will be denoted throughout by $P_{\mathcal K}$, lies  in
$\sigma(\mathcal M)')$, and if $\mathcal K$ (respectively, both $ \mathcal{K}, \mathcal{K}^{\bot}$) is  left invariant by each operator $T (\xi)$ for $\xi \in E.$ Then  the natural restriction of this representation provides a new representation of $E$ on $\mathcal{K}$ and it will denoted by $(\sigma, T)|_{\mathcal{K}}.$
\end{definition}
\begin{definition}
Let $(\sigma, T)$ be a completely bounded, covariant representation of $E$ on a Hilbert space $\mathcal{H}.$ We say that the covariant representation $(\sigma, T)$ admits {\rm Wold-type decomposition} if the representation $(\sigma, T)$ decomposes into a direct sum $(\sigma_1, T_1) \oplus (\sigma_2, T_2)$ on $\mathcal{H}=\mathcal{H}_1 \oplus \mathcal{H}_2$ where $(\sigma_1, T_1)=(\sigma, T)|_{\mathcal{H}_1}$ has generating wandering subspace property and $(\sigma_2, T_2)=(\sigma, T)|_{\mathcal{H}_2}$ is isometric and fully co-isometric covariant representation.
\end{definition}


The following Wold-type decomposition for a completely bounded, covariant representation of a $C^*$-correspondence which are close to isometric is from \cite{SHV16}(see also \cite{MS99}). We use symbol $I$ for $I_{\mathcal H}.$

\begin{theorem}\label{MT1}
Let $(\sigma, T)$ be a completely bounded, covariant representation of $E$ on  $\mathcal{H},$ which satisfies any one of the following conditions:
\begin{enumerate}
\item[(1)]  $(\sigma, T)$ is concave, that is, \\ $\| \wt{T}_2(\eta \otimes h)\|^2+\|\eta \otimes h\|^2 \leq 2 \|(I_E \otimes  \wt{T})(\eta \otimes h)\|^2, \eta \in E^{\otimes 2}, h \in \mathcal{H};$
\item[(2)]  $\|(I_{E} \otimes \wt{T})(\zeta)+\kappa\|^2 \leq 2(\|\zeta\|^2+\|\wt{T}(\kappa)\|^2), \:\:\:\: \zeta \in E^{\otimes 2} \otimes \mathcal{H},\kappa \in E \otimes \mathcal{H};$
\item [(3)]  $\|\wt{T}(\xi)\|^2+\|\wt{T}^*_2\wt{T}(\xi)\|^2 \leq 2\|\wt{T}^*\wt{T}(\xi )\|^2, \:  \xi \in E \otimes \mathcal{H}.$

\end{enumerate}
Then $(\sigma, T)$ admits  {\it Wold-type decomposition}. Moreover, the  decomposition is unique and  the corresponding reducing subspaces are given by
$$\mathcal{H}_1=\bigvee_{n\geq0}\mathfrak{L}_n(\mathcal{W}) \:\:\:\:\: \mathcal{H}_2=\bigcap_{n \geq 1}\wt{T}_n(E^{\otimes n} \otimes \mathcal{H}),$$ where $\mathcal{W}=\mbox{ker}\wt{T}^{*}.$  In particular, if $(\sigma, T)$ is analytic (pure), that is, $\mathcal{H}_{2}=\{0\}$, then $\mathcal{W}=\mathcal{H} \ominus \wt{T}(E \otimes \mathcal{H})$ is the generating wandering subspace for $(\sigma, T),$ that is, $\mathcal{H}=\bigvee_{n \in \mathbb{N}_0}\mathfrak{L}_n(\mathcal{W}).$
\end{theorem}

	

In the case of isometric covariant representations of product systems Solel proved in \cite{S08} that the doubly commuting condition \eqref{doubly} is equivalent to \emph{Nica-covariance} (see \cite{N92}). Our main theorem, Theorem \ref{MT2}, extends Theorem \ref{MT1} for the doubly commuting completely bounded covariant representations which are close to isometric.


\section{Notation and basic definitions for doubly commuting case}

Throughout the paper $k \in \mathbb N.$ The central tool we require is a
product system of $C^*$-correspondences (see \cite{F02,S06, S08, SZ08}): The {\it product system}
$\be$ is defined by a family of $C^*$-correspondences $\{E_1, \ldots,E_k\},$ and by the unitary
isomorphisms $t_{i,j}: E_i \ot E_j \to E_j \ot E_i$ ($i>j$). Using these identifications, for all
${\bf n}=(n_1, \cdots, n_k) \in \Nk$ the correspondence $\be ({\bf n})$ is identified with $E_1^{\ot^{ n_1}} \ot \cdots \ot E_k^{\ot^{n_k}}.$ We use notations $t_{i,i} = \id_{E_i \ot E_i}$ and $t_{i,j} = t_{j,i}^{-1}$ when $i<j.$


\begin{definition}
	Assume $\be$ to be a product system over $\Nk$. Let $\sigma$ be a
	representation of $\mathcal M$ on $\mathcal H$ and let $T^{(i)}:E_i \to B(\mathcal H)$ be a  linear map for $1 \leq i \leq k.$ The tuple $(\sigma, T^{(1)}, \dots, T^{(k)})$ is  called a {\rm completely bounded (respectively, completely contractive) covariant representation} of product system $\mathbb{E}$ on $\mathcal{H}$ if each pair $(\sigma, T^{(i)})$ is a {\rm completely bounded (respectively, completely contractive) covariant representation} of $E_i$ on $\mathcal{H}$ and satisfy the commutation relation
	\begin{equation} \label{rep} \wT^{(i)} (I_{E_i} \ot \wT^{(j)}) = \wT^{(j)} (I_{E_j} \ot \wT^{(i)}) (t_{i,j} \ot I_{\mathcal H})\end{equation}
	with $1\leq i,j\leq k$.
The covariant representation $(\sigma, T^{(1)}, \ldots, T^{(k)})$ is called {\rm isometric} if every $(\sigma, T^{(i)})$ is isometric as a covariant representation of the $C^*$-correspondence $E_i$, and similarly {\rm fully coisometric} representation  is defined.
\end{definition}
\begin{example}\label{Exa1}
Let  $T=(T_1, \dots, T_k)$ be a $k$-tuple of  operators on  Hilbert space $\mathcal{H}$ such that $T_jT_i=q_{ij}T_iT_j$ for all $1\leq i, j \leq k,$ where $q_{ij}$ are non-zero complex numbers, then we shall refer to $(T_1, \dots, T_k)$ as  {\it $q$-commuting} (cf. \cite{BB02, JPS05}). This implies that $q_{ij}=q^{-1}_{ji}$ when $T_iT_j\neq 0,$ where $ i,j \in\{1, \dots ,k\}.$ On the other hand, suppose $ z_{ij} \in \mathbb{T}$ and $z_{ji}=\overline{z}_{ij}$ for all $i \neq j.$ For $1 \leq i,j \leq k$ set $E_i:
=\mathbb{C}$ and define $t_{ij}: E_i \otimes E_j \to E_j \otimes E_i$ as follows:
$$t_{ij}(x \otimes y):=\left\{
\begin{array}{ll}
z_{ij}(y \otimes x), & \hbox{if} \: i \neq j, \\
x \otimes y, & \hbox{if} \:i=j.
\end{array}
\right.$$
Observe that  $t_{ij}$'s are unitaries with $t_{ij}^{-1}=t_{ji}$ for $i \neq j$ and $t_{ii}=id.$  Hence $(\{ E_i \}_{i=1}^k, \{ t_{ij}\}_{i,j=1}^k )$ uniquely determines a product system $\mathbb{E}$ over $\mathbb{N}^k_0.$ Let $1 \leq i \leq k,$ define $T^{(i)}: E_i \to B(\mathcal{H})$ by $T^{(i)}(x)h:=xT_i(h), x \in E_i, h \in \mathcal{H}$ and $\sigma: \mathbb{C} \to B(\mathcal{H})$ by $\sigma(x):=xI _{\mathcal{H}},$ that is, $\sigma$ is identity representation of $\mathbb{C}$ on a Hilbert space $\mathcal{H}.$ It is easy to see that if $(T_1, \dots, T_k)$ is  $q$-commuting, then $(\sigma, T^{(1)}, \dots , T^{(k)})$ is completely bounded covariant representation of $\mathbb{E}$ on the Hilbert space $\mathcal{H}.$ Conversely, let $(\sigma, T^{(1)}, \ldots,T^{(k)})$ be a completely bounded  covariant representation of the product systems $\be$ over $\mathbb{N}^k_0$ on the Hilbert space $\mathcal H,$ where $E_i=\mathbb{C},$ for all $ 1 \leq i \leq n.$  Since $E_i=\mathcal{M}=\mathbb{C}$, $\sigma$ is the identity representation and for each $i \in\{1, \dots, k\},$ the completely bounded linear map $T^{(i)}$ may be considered as a bounded operator on $\mathcal{H.}$ Therefore, the map
\begin{align}\label{qc}
(\sigma, T^{(1)}, \ldots,T^{(k)})\longleftrightarrow (T^{(1)}(1), T^{(2)}(1), \dots, T^{(k)}(1))
\end{align}
induces a bijection between the set of all completely bounded covariant representation of $\mathbb{E}$ (with $E_i= \mathbb{C}$ for all $1 \leq i \leq k$) on a Hilbert space $\mathcal{H}$ over the $C^*$-algebra $\mathbb{C}$ onto the collection of all $q$-commuting operators $(T_1, \dots, T_k)$ on $\mathcal{H}.$
\end{example}

We say that two such a completely bounded covariant representations $(\sigma, T^{(1)}, \ldots, T^{(k)})$ and $(\psi, V^{(1)}, \ldots, V^{(k)})$ of the product system $\be$ over $\mathbb{N}^k_0$,
respectively on Hilbert spaces $\mathcal H$ and $\mathcal K$, are {\it isomorphic}  (cf. \cite{SZ08})  if we have a unitary $U:\mathcal H \to
\mathcal K$ which gives the unitary equivalence of representations $\sigma$ and $\psi,$ and also for each $1\leq i \leq k$, $\xi \in E_i$ one has $S^{(i)}(\xi) = U T^{(i)} (\xi) U^*$.

For each $1\leq i \leq k$
and $l \in \bn$ define $\wT^{(i)}_l: E_i^{\ot l}\otimes \mathcal H\to \mathcal H$ by
\[ \wT^{(i)}_l (\xi_1 \ot \cdots \ot \xi_l \ot h) := T^{(i)} (\xi_1) \cdots T^{(i)}(\xi_l) h\]
 where $\xi_1, \ldots, \xi_l \in E_{i}, h \in \mathcal H$. Then
\begin{equation}\label{eqnn}\wT^{(i)}_l=\wT^{(i)}(I_{E_i} \otimes \wT^{(i)}) \cdots (I_{E^{\otimes l-1}_i} \otimes  \wT^{(i)}).\end{equation}



 For $\mathbf{n}=(n_1, \cdots, n_k) \in \mathbb{N}_0^k $, we write $\wT_{\mathbf{n}}:\mathbb{E}(\mathbf{n})\otimes_{\sigma} \mathcal{H} \to \mathcal{H}$ by
 $$\wT_{\mathbf{n}}=\wT^{(1)}_{n_1}\left(I_{E_1^{\otimes n_1}} \otimes\wT^{(2)}_{n_2}\right) \cdots \left(I_{E_1^{\otimes n_1} \otimes \cdots \otimes E_{k-1}^{\otimes {n_{k-1}}}} \otimes\wT^{(k)}_{n_k}\right).$$
 The map $T_{\mathbf{n}}: \mathbb{E}(\mathbf{n}) \to B(\mathcal{H})$ is then defined by $T_{\mathbf{n}}(\xi)h=\wT_{\mathbf{n}}(\xi \otimes h), ~ \xi \in \mathbb{E}(\mathbf{n}), h \in \mathcal{H}.$

\begin{definition}	
	Let  $\mathcal K$ be a closed subspace of a Hilbert space $ \mathcal H.$  The subspace $\mathcal K$ is called {\rm invariant}({ respectively, \rm reducing}) (cf. \cite{SZ08}) for a covariant representation $(\sigma, T^{(1)}, \ldots, T^{(k)})$ on $\mathcal H,$ if   $ \mathcal{K}$ is $(\sigma, T^{(i)})$-invariant(respectively,  $(\sigma, T^{(i)})$-reducing) subspace for  $1 \leq i \leq k.$  Then it is evident that the natural `restriction' of this representation to $\mathcal K$ provides a new representation
	of $\be$ on $\mathcal K$, which is called a {\rm summand} of $(\sigma, T^{(1)}, \ldots, T^{(k)})$ and will be
	denoted by $(\sigma, T^{(1)}, \ldots, T^{(k)})|_{\mathcal K}$.	
	\end{definition}

 Let $A=\{i_1, \dots i_p \} \subseteq  I_{k},$ where $I_k:=\{1,2, \dots ,k\}$,   denote $\mathbb{N}_0^A:=\{\mathbf{m}=(m_{i_1}, \cdots ,m_{i_p})\: : \: m_{i_j} \in \mathbb{N}_0, 1 \leq j \leq p\}.$ Let $\mathbf{m}=(m_{i_1}, \cdots ,m_{i_p})\in \mathbb{N}_0^A$, define $\wT_{\mathbf{m}}^A:\mathbb{E}(\mathbf{m})\otimes_{\sigma} \mathcal{H} \to \mathcal{H}$ by
 $$\wT_{\mathbf{m}}^A=\wT^{(i_1)}_{m_{i_1}}\left(I_{E_{i_1}^{\otimes m_{i_1}}} \otimes\wT^{(i_2)}_{m_{i_2}}\right) \cdots \left(I_{E_{i_1}^{\otimes m_{i_1}} \otimes \cdots \otimes E_{i_{p-1}}^{\otimes {m_{i_{p-1}}}}} \otimes\wT^{(i_p)}_{m_{i_p}}\right).$$

 Moreover, if $\mathcal K$ is a $\sigma(\mathcal{M})$-invariant subspace,  then we denote $$\mathfrak{L}_{\mathbf{m}}^A(\mathcal{K}):=\bigvee \{T_{m_{i_1}}^{(i_1)}(\eta_{i_1}) \cdots T_{m_{i_p}}^{(i_p)}(\eta_{i_p})h\: : \: \:\: \eta_{i_j} \in E_{i_j}^{\otimes m_{i_j}}, 1 \leq j \leq p , h \in \mathcal{K}\}.$$ Clearly $\mathfrak{L}_{\mathbf{m}}^A(\mathcal{K})=\overline{\wT_{\mathbf{m}}^A(\mathbb{E}(\mathbf{m}) \otimes_{\sigma} \mathcal{K})}.$ It is easy to see that   $\mathfrak{L}_{\mathbf{m}}^A(\mathcal{K})$ is a smallest  $(\sigma,T^{(i_1)}, \dots ,T^{(i_p)})$-invariant subspace which contains $\mathcal{K}.$


\begin{definition}\begin{enumerate}\item
 We say that the $\sigma(\mathcal{M})$-invariant closed subspace $\mathcal{K}$ is said to be {\rm wandering} for  the covariant representation $(\sigma, T^{(i_1)} ,\dots, T^{(i_p)})$ if
$$\mathcal{K}\perp\mathfrak{L}_{\mathbf{m}}^{A}(\mathcal{K})~\mbox{ for each }~\mathbf{m} \in \mathbb{N}_0^{A} \setminus\{0\}.$$
\item
The covariant representation $(\sigma, T^{(i_1)} ,\dots, T^{(i_p)})$  is said to have the {\rm generating wandering subspace  property} if there exists a wandering subspace $\mathcal{K} \subseteq \mathcal{H}$ for $(\sigma, T^{(i_1)} ,\dots, T^{(i_n)})$ such that $[\mathcal{K}]_{T_{A}}=\mathcal{H},$ that is, $$\mathcal{H}=\bigvee_{\mathbf{m} \in \mathbb{N}_0^{A}}\mathfrak{L}_{\mathbf{m}}^{A}(\mathcal{K}).$$
\end{enumerate}
\end{definition}

Let $A=\{i_1, \dots ,i_n \}$ be a non-empty subset of $I_k,$ define the closed subspace $\mathcal{W}_{A}$ of $\mathcal{H}$ by
\begin{align}\label{WA}
\mathcal{W}_{A}:=\bigcap_{j=1}^n (\mathcal{H} \ominus \wt{T}^{(i_j)}(E_j \otimes \mathcal{H})).
\end{align}
Again, if $A=\{i\}$ we simply write $\mathcal{W}_i:=\mathcal{H} \ominus \wt{T}^{(i)}(E_i \otimes \mathcal{H}).$ Therefore $$\mathcal{W}_{A}=\bigcap_{i_j\in A}\mathcal{W}_{i_j}.$$

\begin{definition}  \label{dcom}
	A  completely bounded, covariant representation $(\sigma, T^{(1)}, \ldots,$ $ T^{(k)})$ of  $\be$ on a Hilbert space $\mathcal H$ is said to be {\rm doubly
	commuting}  (cf. \cite{S08}) if for each distinct $i,j \in \{1,\ldots,k\}$ we have
	\begin{equation}\label{doubly}\wT^{(j)^*} \wT^{(i)} =
	(I_{E_j} \ot \wT^{(i)})  (t_{i,j} \ot I_{\mathcal H})  (I_{E_i} \ot \wT^{(j)^*}).
	\end{equation}
\end{definition}

\begin{definition}
	We say that a q-commuting operators $(T_1, \dots,T_k)$ is {\rm $q$-doubly commuting} (cf. \cite{Web,MP}) if
	\begin{align}\label{qb}
		T_i^*T_j=\overline{z_{ij}}T_jT_i^*, \hspace{1cm} \mbox{for all} \:\:  i,j  \:\mbox{with} \:i \neq j.
\end{align}\end{definition}
\begin{example}
	If $T:=(T_1, \dots,T_k)$ is $q$-doubly commuting tuple, the corresponding covariant representation $(\sigma, T^{(1)},$ $ \dots , T^{(k)})$ defined in Example \ref{Exa1} is doubly commuting (for details see \cite[Section 4]{S08}).
\end{example}

For distinct $i, j \in \onek$, a simple calculation (cf. \cite[p. 460]{SZ08}) using Equation \eqref{doubly} yields
\begin{align}\label{eqn*}
\wt{T}^{(i)^*}\wt{T}^{(i)}\wt{T}^{(j)^*}\wt{T}^{(j)} =
\wt{T}^{(j)^*}\wt{T}^{(j)}\wt{T}^{(i)^*}\wt{T}^{(i)} .
\end{align}
Thus the operators $\{I_{\mathcal H}-\wt{T}^{(i)^*}\wt{T}^{(i)}\}_{i=1}^k$ commute to each other. The following proposition is essential in order to extend Theorem \ref{MT1} for the covariant representation $(\sigma, T^{(1)}, \dots, T^{(k)})$  and it follows immediately
from \cite[Proposition 4.6]{HS19}.

\begin{proposition}\label{P21}
Let $\mathbb{E}$ be a product system of $C^*$-correspondences over $\mathbb{N}_0^k$ and let $(\sigma, T^{(1)}, \dots,$ $ T^{(k)})$ be a doubly commuting completely bounded, covariant representation of  $\mathbb{E}$ on a Hilbert space $\mathcal{H}.$ Then for each non-empty subset $\alpha \subseteq I_{k}$, the subspace $\mathcal{W}_{\alpha}$ is $(\sigma, T^{(j)})$-reducing,  where $j \notin \alpha.$ Moreover
$$\mathcal{W}_{\alpha} \ominus \wt{T}^{(j)}(E_j \otimes \mathcal{W}_{\alpha})=\mathcal{W}_{\alpha} \cap \mathcal{W}_j, \:\: \for all \:j \notin \alpha  .$$
\end{proposition}
\section{Shimorin-Wold-type decomposition for a doubly commuting covariant representation of a product system}\label{SecW}

The following theorem is main result of this paper, we use it in the next section, to provide a complete classification of doubly commuting invariant subspaces.
\begin{theorem}\label{MT2}

Let $(\sigma, T^{(1)}, \dots, T^{(k)})$ be a doubly commuting completely bounded, covariant representation of the product system $\mathbb{E}$ on a Hilbert space $\mathcal{H}$  satisfies one of the following properties:
\begin{enumerate}
\item[(1)]  $(\sigma, T^{(i)})$ is concave for each $i=1, \dots, k,$ that is, \\ $\| \wt{T}_2^{(i)}(\eta_i \otimes h)\|^2+\|\eta_i \otimes h\|^2 \leq 2 \|(I_{E_i} \otimes  \wt{T}^{(i)})(\eta_i \otimes h)\|^2, \eta_i \in E^{\otimes 2}_i, h \in \mathcal{H};$
\item[(2)]  for any  $\zeta_i \in E^{\otimes 2}_i \otimes \mathcal{H},\kappa_i \in E_i \otimes \mathcal{H} $

$\|(I_{{E}_i} \otimes \wt{T}^{(i)})(\zeta_i)+\kappa_i\|^2 \leq 2(\|\zeta_i\|^2+\|\wt{T}^{(i)}(\kappa_i)\|^2).$

\item[(3)]   $\|\wt{T}^{(i)}(\xi_i)\|^2+\|\wt{T}^{(i)*}_2\wt{T}^{(i)}(\xi_i)\|^2 \leq 2\|\wt{T}^{(i)*}\wt{T}^{(i)}(\xi_i )\|^2, \:  \xi_i \in E_i \otimes \mathcal{H}.$

\end{enumerate}
 Then for   $2 \leq m \leq k,$ there exists $2^m$  $(\sigma, T^{(1)}, \ldots,	 T^{(m)})$-reducing subspaces $\{\mathcal{H}_{A} \: : \: A \subseteq I_m \}$ such that $\mathcal{H}=\bigoplus_{A \subseteq I_m}\mathcal{H}_A$   and for each  $A=\{i_1, \dots, i_p\} \subseteq I_m: \: (\sigma, T^{(i_1)}, \dots,T^{(i_p)})|_{\mathcal{H}_A} $ has  generating wandering subspace property of the product system $\mathbb{E}_A$ over $\mathbb{N}^A_0$ given by the family of $C^*$-Correspondence $\{E_{i_1}, \dots, E_{i_p}\}$  and $(\sigma, T^{(i)})|_{\mathcal{H}_A}$ is  isometric and fully co-isometric covariant representation whenever $i \in I_m \setminus A.$  Moreover, the above decomposition is unique and
 \begin{align}\label{HA}
  \mathcal{H}_A&=\bigvee_{\mathbf{n} \in\mathbb{N}_0^A}\mathfrak{L}_{\mathbf{n}}^{A}\left(\bigcap_{\mathbf{j} \in \mathbb{N}_0^{I_m \setminus A}} \mathfrak{L}_{\mathbf{J}}^{I_m \setminus A}(\mathcal{W}_A)\right).
  \end{align}
\end{theorem}
\begin{proof}
We shall prove this result by Mathematical induction.

Suppose $m=2:$ Apply Wold-type decomposition, Theorem \ref{MT1} to the covariant representation $(\sigma, T^{(1)}),$ we get
\begin{align*}
 \mathcal{H}&=\bigvee_{n_1 \in \mathbb{N}_0}\widetilde{T}_{n_1}^{(1)}(E_{1}^{\otimes n_1} \otimes \mathcal{W}_{1}) \bigoplus \left( \bigcap_{n_1 \in \mathbb{N}_0}\mbox{ran}(\widetilde{T}_{n_1}^{(1)})\right)\\&=\nonumber \bigvee_{n_1 \in \mathbb{N}_0}\mathfrak{L}^{(1)}_{n_1}(\mathcal{W}_{1}) \bigoplus \left(\bigcap_{n_1 \in \mathbb{N}_0}\mathfrak{L}_{n_1}^{(1)}(\mathcal{H}) \right).
 \end{align*}
Since $\mathcal{W}_{1}$ is a $(\sigma, T^{(2)})$- reducing subspace, by applying the Wold-type decomposition to $(\sigma, T^{(2)})|_{\mathcal{W}_{1}}$, we have
 \begin{align*}
\mathcal{W}_{1}&=\bigvee_{n_2 \in \mathbb{N}_0}\mathfrak{L}^{(2)}_{n_2}(\mathcal{W}_{1}\ominus \mathfrak{L}^{(2)}(\mathcal{W}_{2}) ) \bigoplus \left(\bigcap_{n_2 \in \mathbb{N}_0}\mathfrak{L}_{n_2}^{(2)}(\mathcal{W}_{1})\right)\\&=\bigvee_{n_2 \in \mathbb{N}_0}\mathfrak{L}^{(2)}_{n_2}(\mathcal{W}_{1}\cap \mathcal{W}_{2} ) \bigoplus \left(\bigcap_{n_2 \in \mathbb{N}_0}\mathfrak{L}_{n_2}^{(2)}(\mathcal{W}_{1})\right),
 \end{align*}
 where the second equality follows from Proposition \ref{P21}. Therefore
 \begin{align}\label{Eqn1.1}
 \nonumber \mathcal{H}&=\bigvee_{n_1 \in \mathbb{N}_0}\mathfrak{L}^{(1)}_{n_1}(\mathcal{W}_{1}) \bigoplus \left(\bigcap_{n_1 \in \mathbb{N}_0}\mathfrak{L}_{n_1}^{(1)}(\mathcal{H}) \right).
 \\&=\nonumber \bigvee_{n_1 \in \mathbb{N}_0}\mathfrak{L}^{(1)}_{n_1}\left(\bigvee_{n_2 \in \mathbb{N}_0}\mathfrak{L}^{(2)}_{n_2}(\mathcal{W}_{1}\cap \mathcal{W}_{2} ) \bigoplus \left(\bigcap_{n_2 \in \mathbb{N}_0}\mathfrak{L}_{n_2}^{(2)}(\mathcal{W}_{1})\right) \right) \bigoplus \left(\bigcap_{n_1 \in \mathbb{N}_0}\mathfrak{L}_{n_1}^{(1)}(\mathcal{H}) \right)
 \\&=\bigvee_{\mathbf{n} \in \mathbb{N}^2_0}\mathfrak{L}^A_{\mathbf{n}}(\mathcal{W}_{1} \cap \mathcal{W}_2)\bigoplus \bigvee_{n_1 \in \mathbb{N}_0}\mathfrak{L}^{(1)}_{n_1}\left(\bigcap_{n_2 \in \mathbb{N}_0} \mathfrak{L}^{(2)}_{n_2}(\mathcal{W}_{1})\right)\bigoplus \left(\bigcap_{n_1 \in \mathbb{N}_0}\mathfrak{L}_{n_1}^{(1)}(\mathcal{H}) \right),
 \end{align}
 where $A=\{1,2\}.$
The last equality follows that  $$\bigvee_{n_1 \in \mathbb{N}_0}\mathfrak{L}^{(1)}_{n_1}\left(\bigcap_{n_2 \in \mathbb{N}_0} \mathfrak{L}^{(2)}_{n_2}(\mathcal{W}_{1})\right)=\bigcap_{n_2 \in \mathbb{N}_0} \mathfrak{L}^{(2)}_{n_2}\left(\bigvee_{n_1 \in \mathbb{N}_0}\mathfrak{L}^{(1)}_{n_1}(\mathcal{W}_{1}) \right) \subseteq \bigcap_{n_2 \in \mathbb{N}_0} \mathfrak{L}^{(2)}_{n_2}(\mathcal{H}),$$
$$\bigvee_{\mathbf{n} \in \mathbb{N}^2_0}\mathfrak{L}^A_{\mathbf{n}}(\mathcal{W}_{1} \cap \mathcal{W}_2)=\bigvee_{n_2 \in \mathbb{N}_0}\mathfrak{L}^{(2)}_{n_2}\left(\bigvee_{n_1 \in \mathbb{N}_0}\mathfrak{L}^{(1)}_{n_1}(\mathcal{W}_{1} \cap \mathcal{W}_2) \right) \subseteq \bigvee_{n_2 \in \mathbb{N}_0}\mathfrak{L}^{(2)}_{n_2}( \mathcal{W}_2) .$$
  Applying Theorem \ref{MT1} again for the  covariant representation $(\sigma, T^{(2)}),$ we obtain

 \begin{align*}
 \mathcal{H}&=\nonumber \bigvee_{n_2 \in \mathbb{N}_0}\mathfrak{L}^{(2)}_{n_2}(\mathcal{W}_{2}) \bigoplus \left(\bigcap_{n_2 \in \mathbb{N}_0}\mathfrak{L}_{n_2}^{(2)}(\mathcal{H}) \right),\\
 \end{align*}
 yields
 \begin{align}\label{Eqn1.2}
 \bigcap_{n_1 \in \mathbb{N}_0}\mathfrak{L}_{n_1}^{(1)}(\mathcal{H})= \bigvee_{n_2 \in \mathbb{N}_0}\mathfrak{L}^{(2)}_{n_2}\left(\bigcap_{n_1 \in \mathbb{N}_0}\mathfrak{L}_{n_1}^{(1)}(\mathcal{W}_{2})\right) \bigoplus  \bigcap_{\mathbf{n} \in \mathbb{N}_0^2}\mathfrak{L}_{\mathbf{n}}^{A}(\mathcal{H}),
 \end{align}
 since $\mathcal{W}_2$ is $(\sigma, T^{(1)})$-reducing and   $\bigcap_{\mathbf{n} \in \mathbb{N}_0^2}\mathfrak{L}_{\mathbf{n}}^{A}(\mathcal{H}) \subseteq \bigcap_{n_2 \in \mathbb{N}_0}\mathfrak{L}_{n_2}^{(2)}(\mathcal{H}).$
Now applying Equations (\ref{Eqn1.1}) and (\ref{Eqn1.2}), we get
\begin{align*}
\mathcal{H} &=\bigvee_{\mathbf{n} \in  \mathbb{N}^2_0}\mathfrak{L}^A_{\mathbf{n}} (\mathcal{W}_1 \cap \mathcal{W}_2) \bigoplus  \bigvee_{n_1 \in \mathbb{N}_0} \mathfrak{L}^{(1)}_{n_1} \left(\bigcap_{n_2 \in \mathbb{N}_0}\mathfrak{L}^{(2)}_{(n_2)} (\mathcal{W}_1)\right) \\&\:\:\:\:
\bigoplus \bigvee_{n_2 \in \mathbb{N}_0}\mathfrak{L}^{(2)}_{n_2}\left( \bigcap_{n_1 \in \mathbb{N}_0}\mathfrak{L}^{(1)}_{n_1}(\mathcal{W}_2)  \right)\bigoplus  \bigcap_{\mathbf{n}=  \in \mathbb{N}^2_0}\mathfrak{L}^A_{\mathbf{n}}(\mathcal{H}),~\mbox{where}~A=\{1,2\}.
\end{align*} that is, $$\mathcal{H}=\bigoplus_{A \subseteq I_m}\mathcal{H}_A,$$  where $\mathcal{H}_A $ as in Equation (\ref{HA}) for each $A \subseteq I_2$.

 For the case $m+1\leq n$. Let us assume that for each $m < n$, we have $\mathcal{H}=\bigoplus_{A \subseteq I_m}\mathcal{H}_A$, where  for each non-empty subset $A$ of $I_m$
  \begin{align}
  \mathcal{H}_A&=\bigvee_{\mathbf{n} \in\mathbb{N}_0^A}\mathfrak{L}_{\mathbf{n}}^{A}\left(\bigcap_{\mathbf{j} \in \mathbb{N}_0^{I_m \setminus A}} \mathfrak{L}_{\mathbf{J}}^{I_m \setminus A}(\mathcal{W}_A)\right)
  \end{align}
  and when $A$ is an empty set,
 $$\mathcal{H}_A=\bigcap_{\mathbf{n} \in \mathbb{N}_0^m}\mathfrak{L}_{\mathbf{n}}^{I_m}(\mathcal{H}).$$
 We want to prove this result for  $m+1\leq k,$ that is, $$\mathcal{H}=\bigoplus_{A \subseteq I_{m+1}}\mathcal{H}_A.$$ Since $\mathcal{W}_A$ is $(\sigma, T^{(m+1)})$-reducing subspace for all non-empty subset $A \subseteq I_m,$ Theorem \ref{MT1}  for $(\sigma, T^{(m+1)})|_{\mathcal{W}_A}$ provides us
 \begin{align}
 \mathcal{W}_A &= \nonumber \bigvee_{n_{m+1} \in \mathbb{N}_0}\mathfrak{L}_{n_{m+1}}^{(m+1)}(\mathcal{W}_A \ominus \mathfrak{L}^{(m+1)}(\mathcal{W}_{m+1})) \bigoplus \left(\bigcap_{n_{m+1}\in \mathbb{N}_0}\mathfrak{L}_{n_{m+1}}^{(m+1)}(\mathcal{W}_A) \right)
 \\&= \nonumber \bigvee_{n_{m+1} \in \mathbb{N}_0}\mathfrak{L}_{n_{m+1}}^{(m+1)}(\mathcal{W}_A \cap \mathcal{W}_{m+1}) \bigoplus \left(\bigcap_{n_{m+1}\in \mathbb{N}_0}\mathfrak{L}_{n_{m+1}}^{(m+1)}(\mathcal{W}_A) \right).
 \end{align}

  Note that $\bigcap_{\mathbf{j} \in \mathbb{N}_0^{I_m \setminus A}}\mathfrak{L}_{\mathbf{J}}^{I_m \setminus A}(\mathcal{W}_B) \subseteq \mathcal{W}_{m+1},$ where $B=A\cup \{m+1\}.$ Then $$\bigvee_{\mathbf{n} \in\mathbb{N}_0^B}\mathfrak{L}_{\mathbf{n}}^{B}\left( \bigcap_{\mathbf{j} \in \mathbb{N}_0^{I_m \setminus A}}\mathfrak{L}_{\mathbf{J}}^{I_m \setminus A}(\mathcal{W}_B)\right) \subseteq \bigvee_{n_{m+1} \in \mathbb{N}_0}\mathfrak{L}_{n_{m+1}}^{(m+1)}(\mathcal{W}_{m+1}), $$ $$\bigvee_{\mathbf{n} \in\mathbb{N}_0^A}\mathfrak{L}_{\mathbf{n}}^{A}\left(\bigcap_{\mathbf{j} \in \mathbb{N}_0^{I_{m+1} \setminus A}}\mathfrak{L}_{\mathbf{J}}^{I_{m+1} \setminus A}(\mathcal{W}_A)\right) \subseteq \bigcap_{n_{m+1}\in \mathbb{N}_0}\mathfrak{L}_{n_{m+1}}^{(m+1)}(\mathcal{H}),$$ and hence it follows that

 \begin{align*}
 \mathcal{H}_A &=\bigvee_{\mathbf{n} \in\mathbb{N}_0^A}\mathfrak{L}_{\mathbf{n}}^{A}\left(\bigcap_{\mathbf{j} \in \mathbb{N}_0^{I_m \setminus A}}\mathfrak{L}_{\mathbf{J}}^{I_m \setminus A}(\mathcal{W}_A) \right)
 \\&=\bigvee_{\mathbf{n} \in\mathbb{N}_0^A}\mathfrak{L}_{\mathbf{n}}^{A}\left(\bigcap_{\mathbf{j} \in \mathbb{N}_0^{I_m \setminus A}}\mathfrak{L}_{\mathbf{J}}^{I_m \setminus A}\left(\bigvee_{n_{m+1} \in \mathbb{N}_0}\mathfrak{L}_{n_{m+1}}^{(m+1)}(\mathcal{W}_B) \bigoplus \left(\bigcap_{n_{m+1}\in \mathbb{N}_0}\mathfrak{L}_{n_{m+1}}^{(m+1)}(\mathcal{W}_A) \right)\right) \right)
  \\&=\bigvee_{\mathbf{n} \in\mathbb{N}_0^B}\mathfrak{L}_{\mathbf{n}}^{B}\left( \bigcap_{\mathbf{j} \in \mathbb{N}_0^{I_m \setminus A}}\mathfrak{L}_{\mathbf{J}}^{I_m \setminus A}(\mathcal{W}_B)\right)\bigoplus \left(\bigvee_{\mathbf{n} \in\mathbb{N}_0^A}\mathfrak{L}_{\mathbf{n}}^{A}\left(\bigcap_{\mathbf{j} \in \mathbb{N}_0^{I_{m+1} \setminus A}}\mathfrak{L}_{\mathbf{J}}^{I_{m+1} \setminus A}(\mathcal{W}_A)\right)\right),
\end{align*}

 Using Theorem \ref{MT1}, for the covariant representation $(\sigma, T^{(m+1)}),$ we have

$$ \mathcal{H}=\bigvee_{n_{m+1} \in \mathbb{N}_0}\mathfrak{L}^{(m+1)}_{n_{m+1}}(\mathcal{W}_{m+1}) \bigoplus \left(\bigcap_{n_{m+1} \in \mathbb{N}_0}\mathfrak{L}_{n_{m+1}}^{(m+1)}(\mathcal{H}) \right). $$
When $A$ is an empty set,
\begin{align*}
\mathcal{H}_A &=\bigcap_{\mathbf{n} \in \mathbb{N}_0^A}\mathfrak{L}^{A}_{\mathbf{n}}(\mathcal{H})
\\&=\bigcap_{\mathbf{n} \in \mathbb{N}_0^A}\mathfrak{L}^{A}_{\mathbf{n}}\left(\bigvee_{n_{m+1} \in \mathbb{N}_0}\mathfrak{L}^{(m+1)}_{n_{m+1}}(\mathcal{W}_{m+1}) \bigoplus \left(\bigcap_{n_{m+1} \in \mathbb{N}_0}\mathfrak{L}_{n_{m+1}}^{(m+1)}(\mathcal{H}) \right)\right)
\\&=\bigvee_{n_{m+1} \in \mathbb{N}_0}\mathfrak{L}_{n_{m+1}}^{(m+1)}\left(\bigcap_{\mathbf{n} \in \mathbb{N}^A_0}\mathfrak{L}^{A}_{\mathbf{n}}(\mathcal{W}_{m+1})  \right) \bigoplus \left( \bigoplus_{\mathbf{n}\in \mathbb{N}_0^B}\mathfrak{L}^{B}_{\mathbf{n}}(\mathcal{H})\right).
\end{align*}
It follows from the above orthogonal decomposition of $\mathcal{H}$ that  the covariant representation $(\sigma,T^{(i_1)}, \dots, T^{(i_p)})|_{\mathcal{H}_A}$ has generating wandering subspace property of the product system $\mathbb{E}_A$ over $\mathbb{N}^A_0$ and  $(\sigma, T^{(i)})|_{\mathcal{H}_A}$ is an isometric and fully co-isometric covariant representation for all $i \in I_m \setminus A$ (cf. Theorem \ref{MT1}). The uniqueness also follows immediately from the uniqueness of Theorem \ref{MT1}.
\end{proof}

\begin{remark}\label{RM1}\begin{enumerate}
		\item Let $(\sigma, T^{(1)}, \dots,T^{(k)})$ be a completely bounded covariant representation of $\mathbb{E}$ on  the Hilbert space $\mathcal{H}.$
		Let $\mathcal{W}$ be a wandering subspace for  $(\sigma, T^{(1)}, \dots,T^{(k)})$, that is,
		$\mathcal{W} \:\bot \:\mathfrak{L}_{\mathbf{n}}^{k}(\mathcal{W}), \:\:  \mathbf{n} \in \mathbb{N}_0^k \backslash \{0\}.$  Let $\mathcal{K}$ be the smallest closed $(\sigma, T^{(1)}, \dots,T^{(k)})$-invariant subspace containing $\mathcal{W},$ that is,
		$$\mathcal{K}=\bigvee_{\mathbf{n} \in \mathbb{N}^k_0}\mathfrak{L}_{\mathbf{n}}^{k}(\mathcal{W}).$$ Then the wandering subspace is unique. Indeed, the uniqueness of the wandering subspace follows from $$\mathcal{K} \ominus \left( \sum_{i=1}^k \mathfrak{L}^i(\mathcal{K})\right)=\left( \mathcal{W} \bigoplus \bigvee_{\mathbf{n} \in \mathbb{N}^k_0\setminus \{0\}}\mathfrak{L}_{\mathbf{n}}^{k}(\mathcal{W}) \right) \ominus \left(\bigvee_{\mathbf{n} \in \mathbb{N}^k_0 \setminus \{0\}} \mathfrak{L}_{\mathbf{n}}^{k}(\mathcal{W}) \right) =\mathcal{W}.$$
		Therefore $\mathcal{W}=\bigcap_{i=1}^k \mbox{ker}(I_{E_i} \otimes P_{\mathcal{K}}){\widetilde T}^{(i)^*}|_{\mathcal{K}},$ where $P_{\mathcal{K}}$ is projection of $\mathcal{H}$ on $\mathcal{K}.$
		\item Suppose  $(\sigma, T^{(1)}, \dots,T^{(k)})$ be a  doubly commuting, isometric covariant representation of $\mathbb{E}$ on  the Hilbert space $\mathcal{H}.$ Let $\mathbf{n}=(n_1, \cdots , n_k)$ and $ \mathbf{m}=(m_1,  \cdots ,m_k) \in \mathbb{N}^k_0$ with $\mathbf{n} \neq \mathbf{m}, \:  n_i \neq m_i$ for some $i \in I_k.$ Let  $\xi_{\mathbf{n}} \in E(\mathbf{n}), \eta_{\mathbf{m}} \in E(\mathbf{m}) $ and $h_1, h_2 \in \mathcal{W}.$ Assume $n_i < m_i$ and if $\mathcal{K}$ is in the above equation such that $\mathcal{K}$ is $(\sigma, T^{(1)}, \dots,T^{(k)})$-reducing subspace,  then we have
\begin{align*}
&\langle \wt{T}_{\mathbf{n}}(\xi_{\mathbf{n}} \otimes h_1), \wt{T}_{\mathbf{m}}(\eta_{\mathbf{m}} \otimes h_2)\rangle \\
 &=\langle (I_{E_i^{\otimes n_i}} \otimes \wt{T}^{(i)^*}_{m_i-n_i}\wt{T}_{\mathbf{n}-n_ie_i})(\xi_{\mathbf{n}} \otimes h_1), (I_{E_i^{\otimes m_i}} \otimes \wt{T}_{\mathbf{m}-m_ie_i})(\eta_{\mathbf{m}} \otimes h_2)\rangle \\
 &= \langle \left(I_{E_i^{\otimes n_i}} \otimes (I_{E_i^{\otimes (m_i-ni)}} \otimes \wt{T}_{\mathbf{n}-n_ie_i})(t_{\mathbf{n}-n_ie_i, \mathbf{(n-m)}e_i} \otimes I)(I_{E(\mathbf{n}-n_ie_i)} \otimes \wt{T}^{(i)^*}_{m_i-n_i})\right)(\xi_{\mathbf{n}} \otimes h_1),\\
 &\hspace{1cm} (I_{E_i^{\otimes m_i}} \otimes \wt{T}_{\mathbf{m}-m_ie_i})(\eta_{\mathbf{m}} \otimes h_2)\rangle \\
 &=0   \hspace{3cm} ( \because h_1 \in \mbox{ker}{\widetilde T}^{(i)^*}|_{\mathcal{K}}).
\end{align*}
Thus, $\mathfrak{L}_{\mathbf{n}}^{k}(\mathcal{W}) \bot \mathfrak{L}_{\mathbf{m}}^{k}(\mathcal{W})$ for all distinct $\mathbf{n}, \mathbf{m} \in \mathbb{N}^k_0.$ Hence
$$\mathcal{K}=\bigvee_{\mathbf{n} \in \mathbb{N}^k_0}\mathfrak{L}_{\mathbf{n}}^{k}(\mathcal{W})=\bigoplus_{\mathbf{n} \in \mathbb{N}^k_0}\mathfrak{L}_{\mathbf{n}}^{k}(\mathcal{W}).$$	\end{enumerate}
\end{remark}
The following statement is an analogue of generating wandering subspace property of the covariant representations and we provide a different proof of \cite[Corollary 4.8]{HS19}, which is a generalization of \cite[Corollary 2.4]{CDSS14}.
\begin{corollary}\label{Cor1}
Let $\be$ be a product system of $C^*$-correspondences over $\bn_0^k.$  Let $(\sigma, T^{(1)}, \ldots,T^{(k)})$ be as in Theorem \ref{MT2} such that  $(\sigma, T^{(j)})$   is analytic, for $1 \leq i \leq k.$ Let $\mathcal{K}$ be a $(\sigma, T^{(1)}, \ldots,T^{(k)})$-reducing subspace. Then $(\sigma, T^{(1)}, \ldots,T^{(k)})|_{\mathcal{K}}$ has generating wandering subspace property, $$\mathcal{K}=\bigvee_{\mathbf{n} \in \mathbb{N}_0^k}\mathfrak{L}_{\mathbf{n}}^{I_k}(\mathcal{W}),$$ for some wandering subspace $\mathcal{W}$ for $(\sigma, T^{(1)}, \ldots,T^{(k)})|_{\mathcal{K}}.$ Moreover, $\mathcal{W}$ is unique, in fact $\mathcal{W}=\bigcap_{i=1}^k \mbox{ker}{\widetilde T}^{(i)^*}|_{\mathcal{K}}.$  In particular, $$\mathcal{W}_{I_k}=\bigcap_{i=1}^k\mbox{ker}\widetilde{T}^{(i)^*},$$  is a generating wandering subspace for $(\sigma, T^{(1)}, \ldots,T^{(k)}),$  that is, $$\mathcal{H}=\bigvee_{\mathbf{n} \in \mathbb{N}_0^k}\mathfrak{L}_{\mathbf{n}}^{I_k}(\mathcal{W}_{I_k}).$$
\end{corollary}
\begin{proof}
Since $(\sigma, T^{(1)}, \ldots,T^{(k)})$ is analytic, then for each $i \in I_k$  $$\bigcap_{n_i \in \mathbb{N}_0}\mathfrak{L}_{n_i}^i(\mathcal{K}) \subseteq \bigcap_{n_i \in \mathbb{N}_0}\mathfrak{L}_{n_i}^i(\mathcal{H})=\{0\}.$$
This implies that, $(\sigma, T^{(1)}, \ldots,T^{(k)})|_{\mathcal{K}}$ is analytic and it satisfies hypothesis of the Theorem \ref{MT2}. Since $\mathcal{K}$ is $(\sigma, T^{(1)}, \ldots,T^{(k)})$-reducing subspace, without loss of generality we can assume that $\mathcal{K}=\mathcal{H}.$ Let $A$ be the proper subset of $I_k$ and let $i \in A \setminus I_k,$ we have
 $$\bigcap_{\mathbf{j} \in \mathbb{N}_0^{I_m \setminus A}} \mathfrak{L}_{\mathbf{J}}^{I_m \setminus A}(\mathcal{W}_A) \subseteq\bigcap_{n_i \in \mathbb{N}_0}\mathfrak{L}_{n_i}^i(\mathcal{H})=\{0\}, $$ where $\mathcal{W}_A$ as in Equation (\ref{WA}).
 Apply previous equation to  Equation (\ref{HA}) in  Theorem \ref{MT2}, we get $\mathcal{H}_A=\{0\}.$  Therefore, if $A$ is empty set then $\mathcal{H}_A=\mathcal{H}.$
 This completes the proof.
 \end{proof}

The following definition  of induced representation is a generalization of the multiplication operators $M_{z_i} \otimes I_{\mathcal{H}}$ on the vector valued Hardy space $H^2_{\mathcal{H}}(\mathbb{D}^k).$
\begin{definition} Let $\mathbb{E}$ be the product system over $\mathbb{N}^k_0,$ and let $\pi$ be a representation of $\mathcal{M}$ on a Hilbert space $\mathcal{K}.$ Define the {\rm Fock module} of  $\mathbb{E},$
$$\mathcal{F}(\mathbb{E}):=\bigoplus_{\mathbf{n} \in \mathbb{N}^k_0}\mathbb{E}(\mathbf{n}).$$ Note that $\mathcal{F}(\mathbb{E})$ is a $C^*$ correspondence over $\mathcal{M}.$  Define a completely contractive covariant representation $(\rho, S^{(i)})$ of $E_i$ on the Hilbert space $\mathcal{F}(\mathbb{E})\otimes_{\pi} \mathcal{K}$ (cf. \cite{SZ08}) by  $$\rho(a):=\phi_{\infty}(a) \otimes I_{\mathcal{K}}, \:\:a \in \mathcal{M}$$ and $$S^{(i)}(\xi_i)=T_{{\xi}_i} \otimes I_{\mathcal{K}},\: i\in \{1, \cdots,k\}, \xi_i \in E_i,$$ where $T_{\xi_i}$  denotes the creation operator on $\mathcal{F}(\mathbb{E})$ determined by $\xi_i,$ that is, $T_{\xi_i}(\eta)=\xi_i \otimes \eta,$ where $\eta \in \mathcal{F}(\mathbb{E})$ and  $\phi_{\infty}$  denotes the canonical left action of $\mathcal{M}$ on $\mathcal{F}(\mathbb{E}).$ It is easy to see that the above representation $(\rho,S^{(1)}, \dots ,S^{(k)})$ is doubly commuting isometric covariant representation of $\mathbb{E}$ on the Hilbert space $\mathcal{F}(\mathbb{E})\otimes_{\pi} \mathcal{K}$ and it is called {\rm induced representation} of $\mathbb{E}$ induced by $\pi.$ Any covariant representation of $\mathbb{E}$ which is isomorphic to $(\rho,S^{(1)}, \dots ,S^{(k)})$ is called an {\rm induced representation}.\end{definition}

The following corollary is a generalization of \cite[Theorem 2.4]{SZ08}:

\begin{corollary}\label{main1}
	Let $\be$ be a product system of $C^*$-correspondences over $\bn_0^k.$ Let $(\sigma, T^{(1)}, \ldots,
	T^{(k)})$ be a doubly commuting isometric covariant representation of $\be$ on a Hilbert space $\mathcal H.$  Then for   $2 \leq m \leq k,$ there exists $2^m$  $(\sigma, T^{(1)}, \ldots,	T^{(m)})$-reducing subspaces $\{\mathcal{H}_{\mathcal{A}} \: : \: A \subseteq I_{m} \}$ such that $\mathcal{H}=\bigoplus_{A \subseteq I_m}\mathcal{H}_A$   and for each  $A=\{i_1, \dots, i_p\} \subseteq I_m: \: (\sigma, T^{(i_1)}, \dots,T^{(i_p)})|_{\mathcal{H}_A} $  is an induced representation  of the product system $\mathbb{E}_A$ over $\mathbb{N}^A_0$ given by the family of $C^*$-correspondence $\{E_{i_1}, \dots, E_{i_p}\}$  and $(\sigma, T^{(i)})|_{\mathcal{H}_A}$ is  isometric and fully co-isometric covariant representation whenever $i \in I_m \setminus A.$  Moreover, the above decomposition is unique and
 \begin{align}\label{HA1}
  \mathcal{H}_A&=\bigoplus_{\mathbf{n} \in\mathbb{N}_0^A}\mathfrak{L}_{\mathbf{n}}^{A}\left(\bigcap_{\mathbf{j} \in \mathbb{N}_0^{I_m \setminus A}} \mathfrak{L}_{\mathbf{J}}^{I_m \setminus A}(\mathcal{W}_A)\right).
  \end{align}
\end{corollary}
\begin{proof}
	If $(\sigma, T^{(1)}, \dots,T^{(k)})$ be an isometric doubly commuting covariant representation of $\mathbb{E}$ on  the Hilbert space $\mathcal{H},$ then it satisfies hypothesis of the Theorem \ref{MT2}.  By (2) of the Remark \ref{RM1}, then  $\mathcal{H}_A$ as in Theorem \ref{MT2} can be written as
	\begin{align*}
	\mathcal{H}_A&=\bigvee_{\mathbf{n} \in\mathbb{N}_0^A}\mathfrak{L}_{\mathbf{n}}^{A}\left(\bigcap_{\mathbf{j} \in \mathbb{N}_0^{I_m \setminus A}} \mathfrak{L}_{\mathbf{J}}^{I_m \setminus A}(\mathcal{W}_A)\right)=\bigoplus_{\mathbf{n} \in\mathbb{N}_0^A}\mathfrak{L}_{\mathbf{n}}^{A}\left(\bigcap_{\mathbf{j} \in \mathbb{N}_0^{I_m \setminus A}} \mathfrak{L}_{\mathbf{J}}^{I_m \setminus A}(\mathcal{W}_A)\right).
	\end{align*}
Moreover, $\bigcap_{\mathbf{j} \in \mathbb{N}_0^{I_m \setminus A}} \mathfrak{L}_{\mathbf{J}}^{I_m \setminus A}(\mathcal{W}_A)=\bigcap_{i \in A}\mbox{ker}\wt{T}^{(i)^*}|_{\mathcal{H}_A}$ and  is $\sigma$-invariant.
Define a new representation $\sigma_1$ of $\mathcal{M}$ on the Hilbert space $\mathcal{K}$ by $\sigma_1(a)=\sigma(a)|_{\mathcal{K}},$ where $\mathcal{K}=\bigcap_{i \in A}\mbox{ker}\wt{T}^{(i)^*}|_{\mathcal{H}_A}.$ Let $ A=\{i_1, i_2, \dots, i_p\} \subseteq I_m$ and let $(\rho, S^{(i_1)}, \dots, S^{(i_p)})$ be an induced representation of the product system $\mathbb{E}_A$ over $\mathbb{N}^A_0$ induced by $\sigma_1.$
  Define the operator $U: \mathcal{F}(\mathbb{E}_A) \otimes_{\sigma_1}\mathcal{K}  \to \mathcal{H}_A$ by $U(\xi_{\mathbf{n}} \otimes h)=\widetilde{T}^A_{\mathbf{n}}(\xi_{\mathbf{n}} \otimes h), \:\xi_{\mathbf{n}} \in \mathbb{E}_A (\mathbf{n})$ and $h \in \mathcal{K}.$ This shows that $U$ is unitary operator  and
\begin{align*}
US^{(i_j)}(\eta_{i_j})(\xi_{\mathbf{n}} \otimes h)&=U(\eta_{i_j} \otimes \xi_{\mathbf{n}} \otimes h)={T}^{(i_j)}(\eta_{i_j})U(\xi_{\mathbf{n}} \otimes h)\\
U\rho(a)(\xi_{\mathbf{n}} \otimes h)&=U(\phi^n(a)\xi_{\mathbf{n}} \otimes h )=\sigma(a)U(\xi_{\mathbf{n}} \otimes h),
\end{align*}
where $\phi_{\mathbf{n}}$ is left action of $E_A({\mathbf{n}})$ and,  for all $\eta_{i_j} \in E_{i_j},\: \xi_{\mathbf{n}} \in E_A(\mathbf{n}) ,h \in \mathcal{K}$ and $1 \leq j \leq p.$ Hence $(\sigma, T^{(i_1)}, \dots, T^{(i_p)})|_{\mathcal{H}_A}$ is an induced representation.
\end{proof}

From the above Corollary, if $(\sigma, T)$ be an isometric covariant representation of $E$ on the Hilbert space $\mathcal{H}$. Then  $(\sigma, T)$ is analytic if and only if $(\sigma,T)$ is an induced representation. Let $(\sigma, T^{(1)}, \ldots,T^{(k)})$ be an induced representation of  $\be$, then for each $i, \: 1 \leq i \leq k,$ $(\sigma, T^{(i)})$ is also an induced representation. But, the converse not true. The following theorem gives a conceptual characterization of induced representation   and generating wandering subspace for the isometric covariant representation of the product system $\mathbb{E}$ over $\mathbb{N}^k_0, \:  k \geq 2.$
 \begin{theorem}\label{MT3}
 Let $(\sigma, T^{(1)}, \ldots,T^{(k)})$ be an isometric covariant representation of the product systems $\be$ over $\mathbb{N}^k_0$ on the Hilbert space $\mathcal H.$ Then the following conditions are equivalent:
\begin{enumerate}
  \item There exists a wandering subspace $\mathcal{W}$ for $(\sigma, T^{(1)}, \ldots,T^{(k)})$ such that $$\mathcal{H}=\bigoplus_{\mathbf{n} \in \mathbb{N}^k_0}\mathfrak{L}_{\mathbf{n}}^{I_k}(\mathcal{W}).$$
  \item For every $j \in I_k,$ $(\sigma, T^{(j)})$ is an induced representation of $E_j$ and $(\sigma, T^{(1)}, \ldots,T^{(k)})$ is doubly commuting.
   \item There exists $j \in I_k$ such that $(\sigma, T^{(j)})$ is an induced representation and the wandering subspace for $(\sigma, T^{(j)})$ is
     $$\mathcal{W}_j=\bigoplus_{\mathbf{n} \in \mathbb{N}^k_0, n_j=0}\mathfrak{L}_{\mathbf{n}}^{I_k}\left(\bigcap_{i=1}^k\mathcal{W}_i\right).$$
   \item  $\mathcal{W}_{I_k}$ is a wandering subspace for  $(\sigma, T^{(1)}, \ldots,T^{(k)})$ and \\ $\mathcal{H}=\bigoplus_{\mathbf{n} \in \mathbb{N}^k_0}\mathfrak{L}_{\mathbf{n}}^{I_k}(\mathcal{W}_{I_k}).$
   \item $(\sigma, T^{(1)}, \ldots,T^{(k)})$ is  isomorphic to an induced representation $(\rho, S^{(1)}, $ $\cdots ,S^{(k)})$  induced by some representation $\pi$ on $\mathcal{K}$  with  $dim \mathcal{K}=dim \mathcal{W}_{I_k}$.\end{enumerate}
 \end{theorem}
 \begin{proof}
 (1) $\implies$ (2): Observe that, for each $j, 1\leq j \leq k$,
 \begin{align}\label{eqnn1}
 \mathcal{H}=\bigoplus_{\mathbf{n} \in \mathbb{N}^k_0}\mathfrak{L}_{\mathbf{n}}^{I_k}(\mathcal{W})=\bigoplus_{n \in \mathbb{N}_0} \mathfrak{L}_n^{j}\left(\bigoplus_{\mathbf{n} \in \mathbb{N}^k_0, n_j=0}\mathfrak{L}_{\mathbf{n}}^{I_k}(\mathcal{W})\right).
  \end{align}
It implies that $(\sigma, T^{(j)})$ is an induced representation. Let $h \in \mathcal{H}$ such that
 $$h=\sum_{n=0}^{\infty} \wT_n^{(j)}(\xi^j_n \otimes h_n), ~~\mbox{where}~\xi_n^j \in E_j^{\otimes n}, h_n \in\bigoplus_{\mathbf{n} \in \mathbb{N}^k_0, n_i=0}\mathfrak{L}_{\mathbf{n}}^{I_k}(\mathcal{W}).$$ Then for all $i\neq j$ and $\xi_i \in E_i$, we have
 \begin{align*}
 &(I_{E_j} \otimes \wT^{(i)})(t_{i,j} \otimes I)(I_{E_i} \otimes \wT^{(j)^*})(\xi_i\otimes h)
\\&=(I_{E_j} \otimes \wT^{(i)})(t_{i,j} \otimes I)(I_{E_i} \otimes \wT^{(j)^*})\sum_{n=0}^{\infty}(\xi_i \otimes \wT_n^{(j)}(\xi^j_n \otimes h_n))
\\&=\sum_{n=1}^{\infty}(I_{E_j} \otimes \wT^{(i)})(t_{i,j} \otimes I)(I_{E_i} \otimes \wT^{(j)^*})(\xi_i \otimes \wT_n^{(j)}(\xi^j_n \otimes h_n))
\\&=\sum_{n=1}^{\infty}(I_{E_j} \otimes \wT^{(i)})(t_{i,j} \otimes I)(I_{E_i} \otimes I_{E_j} \otimes \wT_{n-1}^{(j)})(\xi_i \otimes \xi^j_n \otimes h_n)) ~\mbox{(using Equation(\ref{eqnn}))}
\\&=\sum_{n=1}^{\infty}\wT^{(j)^*}\wT^{(i)}(I_{E_i}  \otimes \wT_{n}^{(j)})(\xi_i \otimes \xi^j_n \otimes h_n))~\mbox{(using Equations (\ref{rep}),  (\ref{eqnn}))}
\\&=\wT^{(j)^*}\wT^{(i)}(\xi_i\otimes h))~\mbox{(using Equation (\ref{eqnn1}))}.
 \end{align*}
Hence $(\sigma, T^{(1)}, \ldots,T^{(k)})$ is doubly commuting.\\
 (2)$\implies$ (3): By Corollary \ref{Cor1} we obtain $$\mathcal{H}=\bigoplus_{\mathbf{n} \in \mathbb{N}^k_0}\mathfrak{L}_{\mathbf{n}}^{I_k}\left(\bigcap_{i=1}^n\mathcal{W}_i\right)=\bigoplus_{n \in \mathbb{N}_0} \mathfrak{L}_n^{j}\left(\bigoplus_{\mathbf{n} \in \mathbb{N}^k_0, n_j=0}\mathfrak{L}_{\mathbf{n}}^{I_k}\left(\bigcap_{i=1}^k\mathcal{W}_i\right)\right),$$ and hence (3) follows.\\
 (3) $\implies$ (4): Given that $(\sigma, T^{(j)})$ is an induced representation with the wandering subspace $$\mathcal{W}_j=\bigoplus_{\mathbf{n} \in \mathbb{N}^k_0, n_j=0}\mathfrak{L}_{\mathbf{n}}^{I_k}\left(\bigcap_{i=1}^k\mathcal{W}_i\right).$$ It follows that $$\mathcal{H}=\bigoplus_{n \in \mathbb{N}_0} \mathfrak{L}_n^{j}(\mathcal{W}_j)=\bigoplus_{n \in \mathbb{N}_0} \mathfrak{L}_n^{j}\left(\bigoplus_{\mathbf{n} \in \mathbb{N}^k_0, n_j=0}\mathfrak{L}_{\mathbf{n}}^{I_k}\left(\bigcap_{i=1}^k\mathcal{W}_i\right)\right)=\bigoplus_{\mathbf{n} \in \mathbb{N}^k_0}\mathfrak{L}_{\mathbf{n}}^{I_k}\left(\bigcap_{i=1}^k\mathcal{W}_i\right),$$ and hence (4) follows.\\
 (4) $\implies$(5): Define $\sigma_0:={\sigma}|_{\mathcal{W}_{I_k}}$ and define the unitary operator  $$U: \mathcal{H}\left(=\bigoplus_{\mathbf{n} \in \mathbb{N}^k_0} \wT_{\mathbf{n}} \left(E(\mathbf{n}) \otimes_{\sigma_0} \mathcal{W}_{I_k}\right)\right) \to \mathcal{F}\mathbb(E) \otimes_{\sigma_0} \mathcal{W}_{I_k},$$ by
$$U(\sum_{\mathbf{n} \in \mathbb{N}_0^k}\wT_{\mathbf{n}}(w_{\mathbf{n}})):=\sum_{\mathbf{n} \in \mathbb{N}_0^k}w_{\mathbf{n}},~\mbox{where}~w_{\mathbf{n}} \in  E(\mathbf{n}) \otimes_{{\sigma}_0} \mathcal{W}_{I_k}.$$ Then it is easy to see that $$UT^{(j)}(\xi_j)=S^{(j)}(\xi_j)U, ~U\sigma(a)=\rho(a)U$$ for every $\xi_j \in E_j,a \in \mathcal{M}, j \in I_k,$  where  $(\rho, S^{(1)}, \ldots,S^{(k)})$ is an induced representation induced by $\sigma_0.$\\
 (5)$\implies$ (1): is obvious.
 \end{proof}

\begin{definition}
Let $T=(T_1, \dots, T_k)$ be a $q$-commuting  operators on  Hilbert space $\mathcal{H}$ and $1\leq m \leq k.$ Let $A=\{i_1, \dots i_p \} \subseteq  I_{m}.$ We denote by $T_A$ the $|A|$- tuple of $q$-commuting operators $(T_{i_1}, \cdots ,T_{i_p})$ and also denote $T_{i_1}^{n_{i_1}} \cdots T_{i_p}^{n_{i_p}}$ by $T^{\mathbf{n}}_A,$ where $\mathbf{n}=(n_{i_1}, \cdots, n_{i_p}) \in \mathbb{N}_0^A.$ For a closed subspace $\mathcal{K}$ of $\mathcal{H},$ we write $[\mathcal{K}]_{T_A}$ to denote the smallest closed joint $T_A$-invariant subspace of the Hilbert space $\mathcal{H}$ containing  $\mathcal{K},$ that is, $[\mathcal{K}]_{T_A}=\bigvee_{n \in \mathbb{N}_0^A}T_A^{\mathbf{n}}\mathcal{K}.$ A closed subspace $\mathcal{W}$ of $\mathcal{H}$ is said to be {\rm generating wandering subspace} for an $A$-tuple $T_A=(T_{i_1}, \cdots ,T_{i_p})$  of $q$-commuting operators on $\mathcal{H}$ if $\mathcal{W}\: \bot \: T_A^{\mathbf{n}}(\mathcal{W})$ for all $\mathbf{n} \in \mathbb{N}_0^A$ and $\mathcal{H}=[\mathcal{W}]_{T_A}.$ If there exist a closed subspace $\mathcal{W}$ of $\mathcal{H}$ such that $\mathcal{W}$ is generating subspace for $T,$   the tuple $T$ is said to have the {\rm generating wandering subspace property}.
\end{definition}

 Let $(\sigma, T^{(1)}, \ldots,T^{(k)})$ be an completely bounded  covariant representation of the product systems $\be$(with $E_i= \mathbb{C}$ for all $1 \leq i \leq k$) over $\mathbb{N}^k_0$ on the Hilbert space $\mathcal H.$ From the Equation (\ref{qc}) it implies that the $k$-tuple $T:=(T^{(1)}(1), T^{(2)}(1), \dots, T^{(k)}(1))$ is $q$-commuting. Let $\mathcal{K}$ be a closed subspace of $\mathcal{H},$ then it is  easy to see that $$\overline{T_A^{\mathbf{n}}\mathcal{K}}=\mathfrak{L}^A_{\mathbf{n}}(\mathcal{K}),$$ where $A \subseteq I_m$ and $\mathbf{n} \in \mathbb{N}_0^A.$
  This combined with the following corollary, which is a generalization of  \cite[Theorem 3.4]{MP} and \cite[Theorem 2.4]{S14}.
\begin{corollary}\label{OV}
Let $T=(T_1, \dots,T_k)$ be a $k$-tuple of $q$-doubly commuting operators on Hilbert space $\mathcal{H}$ such that $T$ satisfies one of the following properties:
\begin{enumerate}
\item $T_i$ is concave for each $i=1,  \dots , k,$ that is,\\ $\|T_i^2h\|^2+\|h\|^2 \leq 2\|T_ih\|^2, \: h \in \mathcal{H},$
\item $\|T_ih+g\|^2 \leq 2(\|h\|^2+\|T_ig\|^2),$ \: $h, g \in\mathcal{H}$ and  for $i=1, \dots ,k,$
\item $\|T_ih\|^2+\|T_i^{*2}T_ih\|^2 \leq 2\|T_i^*T_ih\|^2, \: h \in \mathcal{H}$ and for $i=1, \dots ,k.$
\end{enumerate}
For $2 \leq m \leq k,$ there exist $2^m$ joint $(T_1, \dots, T_m)$-reducing subspaces $\{\mathcal{H}_A\: : \: A \subseteq I_m\: \}$ such that $\mathcal{H}=\bigoplus_{A \subseteq I_m}\mathcal{H}_A$
and for all $A=\{i_1, \dots, i_p\}\subseteq I_m $ and $\mathcal{H}_A \neq \{0\}; \:\: (T_{i_1}, \dots, T_{i_p})|_{{\mathcal{H_A}}}$ has generating wandering subspace property  and  $T_i|_{{\mathcal{H_A}}}$ is unitary whenever $i \in I_m\setminus A.$  Moreover, the above decomposition is unique and
$$ \mathcal{H}_A=\bigvee_{\mathbf{n} \in\mathbb{N}_0^A}T_A^{\mathbf{n}}\left(\bigcap_{\mathbf{j} \in \mathbb{N}_0^{I_m \setminus A}} T_{I_m \setminus A}^{\mathbf{J}}\mathcal{W}_A\right),$$
 where $\mathcal{W}_A=\bigcap_{i \in A} (\mathcal{H} \ominus T_i\mathcal{H}).$

\end{corollary}
Using the above corollary we can easily prove the results  \cite[Corollary 2.4]{CDSS14}, \cite[Theorem 3.4]{MP} and \cite[Theorem 2.4]{S14}  in the following remark:
\begin{remark}\label{RM3}
(1) Let $T=(T_1, \dots,T_k)$ be as in the above Corollary \ref{OV} such that each $T_i$
 is analytic, that is,  $\bigcap_{n_i=1}^{\infty}T_i^{n_i}\mathcal{H}=\{0\}$ for $1 \leq i \leq k.$ From the Corollary \ref{Cor1}, we have that $$\mathcal{H}=\bigvee_{\mathbf{n} \in \mathbb{N}^k_0}T^{\mathbf{n}}\mathcal{W}_{I_k},$$
  that is, $T$ has a generating wandering subspace property.\\
  (2)  If $T=(T_1, \dots,T_k)$ such that each $T_i$ is an isometry which satisfy  $T_i^*T_j=\overline{{z}_{ij}}T_jT_i^*,$ where $z_{ij}\in  \mathbb{T}$ with $i \neq j, \:  1 \leq i, j \leq k.$  Then it satisfy hypothesis of the Corollary \ref{OV}.
   Using (2) of the Remark \ref{RM1}, we conclude that
    $$\mathcal{H}_A=\bigvee_{\mathbf{n} \in\mathbb{N}_0^A}T_A^{\mathbf{n}}\left(\bigcap_{\mathbf{j} \in \mathbb{N}_0^{I_m \setminus A}} T_{I_m \setminus A}^{\mathbf{J}}\mathcal{W}_A\right)=\bigoplus_{\mathbf{n} \in\mathbb{N}_0^A}T_A^{\mathbf{n}}\left(\bigcap_{\mathbf{j} \in \mathbb{N}_0^{I_m \setminus A}} T_{I_m \setminus A}^{\mathbf{J}}\mathcal{W}_A\right),$$ where $\mathcal{H}_A$ and $\mathcal{W}_A$ as in the Corollary \ref{OV} and $T_i|_{\mathcal{H}_A}$ is pure isometry if $i \in A.$ In particular, if all $T_i$ is analytic, then $\mathcal{H}=\bigoplus_{\mathbf{n} \in \mathbb{N}^k_0}T^{\mathbf{n}}\mathcal{W}_{I_k}.$  Moreover,   the $k$-tuple $(T_1, \dots, T_k)$ of operators   on $\mathcal{H}$ is isomorphic to the multiplication operator $M_{\mathbf{z}}:=(M_{z_1}, \dots, M_{z_k})$  on $\mathcal{H}_{\mathcal{W}}^2(\mathbb{D}^k),$ where $\mathcal{H}_{\mathcal{W}}^2(\mathbb{D}^k)$ denote the Hardy space of $\mathcal{W}(:=\mathcal{W}_{I_k})$-valued analytic functions on $\mathbb{D}^k$ and  $\mathbb{D}$ is an open unit disk in $\mathbb{C}.$ Indeed,  define a unitary operator $$U:\mathcal{H}_{\mathcal{W}}^2(\mathbb{D}^k)=\bigoplus_{\mathbf{n} \in \mathbb{N}^k_0}\mathbf{z}^{\mathbf{n}}\mathcal{W} \to \mathcal{H}=\bigoplus_{\mathbf{n} \in \mathbb{N}^k_0}T^{\mathbf{n}}\mathcal{W},$$ by $U(\sum_{\mathbf{n} \in \mathbb{N}_0^k}\mathbf{z}^{\mathbf{n}}h_{\mathbf{n}})=\sum_{\mathbf{n} \in \mathbb{N}_0^k}T^{\mathbf{n}}h_{\mathbf{n}}$ for all $h_{\mathbf{n}} \in \mathcal{W}.$ It is easy to verify that,  if $z_{ij}=1$ for all $1\leq i,j \leq k,$ then $UM_{z_i}=T_iU$ for $1 \leq i \leq k.$  Conversely, if the multiplication operator $M_{\mathbf{z}}$ and the doubly commuting isometries $T$ are isomorphic then each $T_i$ is analytic, $1 \leq i \leq k$.
   \end{remark}

\section{Beurling-Lax type decomposition for Doubly commuting invariant subspaces}
In this section we completely characterize the  doubly commuting  subspaces of covariant representations of the product systems $\mathbb{E}$ on a Hilbert space $\mathcal{H}.$ This result is a generalization of a version of the Beurling's theorem for the doubly commuting shift on the bidisc by Mandrekar \cite{M88} and the polydisc case proved by Sarkar, Sasane and Wick \cite{S13}.
We first prove Beurling's theorem for the covariant representation of a $C^*$-correspondence which extends Beurling-Lax type theorem due to Popescu in \cite{PoB}(see also a relevant tensor algebra version of Beurling's theorem by Muhly and Solel \cite{MS99}).

\begin{definition}
Let $\mathbb{E}$ be a product system of $C^*$-correspondences  over $\mathbb{N}_0^k$. Let $(\sigma, T^{(1)}, \dots, T^{(k)})$ and $(\psi, V^{(1)}, \dots, V^{(k)})$ be a completely bounded  covariant representations of $\mathbb{E}$ on the Hilbert spaces $\mathcal{H}$ and $\mathcal{K}$ respectively. A bounded operator $A: \mathcal{H}\to \mathcal{K}$ is called {\rm multi-analytic} if it satisfies the  following condition
\begin{align*}
AT^{(i)}(\xi_i)h=V^{(i)}(\xi_i)Ah \hspace{1cm} A \sigma(a)h= \psi(a)Ah,
\end{align*}
where $\xi_i \in E_i, \: h \in \mathcal{H}$ and $ a \in \mathcal{M},  1 \leq i \leq k.$
\end{definition}
Throughout this section, we assume $(\sigma, T^{(1)}, \dots, T^{(k)})$ and $(\psi, V^{(1)}, \dots, V^{(k)})$ to be the doubly commuting isometric  covariant representations of $\mathbb{E}$ on the Hilbert spaces $\mathcal{H}$ and $\mathcal{K}$ such that  for each $1 \leq i \leq k,$ $(\sigma, T^{(i)})$ and $(\psi, V^{(i)})$ are analytic. Then by  Corollary \ref{Cor1}, we have
$$\mathcal{H}=\bigoplus_{\mathbf{n} \in \mathbb{N}_0^k}\wt{T}_{\mathbf{n}}(E(\mathbf{n}) \otimes \mathcal{W}_{\mathcal{H}} ) \:\: \mbox{and} \:\: \mathcal{K}=\bigoplus_{\mathbf{n} \in \mathbb{N}_0^k}\wt{V}_{\mathbf{n}}(E(\mathbf{n}) \otimes \mathcal{W}_{\mathcal{K}} ),$$ where $\mathcal{W}_{\mathcal{H}}$ and $\mathcal{W}_{\mathcal{K}}$ are the generating wandering subspaces for $(\sigma, T^{(1)}, \dots, T^{(k)})$ and $(\psi, V^{(1)},$ $ \dots, V^{(k)}).$
\begin{notation}
If $A: \mathcal{H} \to \mathcal{K}$ is  multi-analytic operator, then $A$ is uniquely determined by the operator $\theta: \mathcal{W}_{\mathcal{H}} \to \mathcal{K},$ where $ \theta:= A|_{\mathcal{W}_{\mathcal{H}}}.$ This follows  because for every $\xi_{\mathbf{n}} \in E(\mathbf{n}), h \in \mathcal{W}_{\mathcal{H}}$ we have $AT_{\mathbf{n}}(\xi_{\mathbf{n}})h=V_{\mathbf{n}}(\xi_{\mathbf{n}})\theta h$ and $\mathcal{H}=\bigoplus_{\mathbf{n} \in \mathbb{N}_0^k}\wt{T}_{\mathbf{n}}(E(\mathbf{n}) \otimes \mathcal{W}_{\mathcal{H}} ).$

Now, let us consider an operator $\theta: \mathcal{W}_{\mathcal{H}} \to \mathcal{K}\left( = \bigoplus_{\mathbf{n} \in \mathbb{N}_0^k}\wt{V}_{\mathbf{n}}(E(\mathbf{n}) \otimes \mathcal{W}_{\mathcal{K}} ) \right).$ We define the operator $M_{\theta}: \mathcal{H} \to \mathcal{K}$ by the relation $$M_{\theta}T_{\mathbf{n}}(\xi_{\mathbf{n}})h=V_{\mathbf{n}}(\xi_{\mathbf{n}})\theta h=V_{\mathbf{n}}(\xi_{\mathbf{n}})M_{\theta} h \:\:\: (\xi_{\mathbf{n}} \in E(\mathbf{n}), h \in \mathcal{W}_{\mathcal{H}}).$$ \\
In this section we only work with $\theta$ such that $M_{\theta}$ is a contraction. It is easy to see that
$$M_{\theta}\left(\bigoplus_{\mathbf{n} \in \mathbb{N}^k_0}h_{\mathbf{n}} \right)=\sum_{\mathbf{n} \in \mathbb{N}^k_0} \wt{V}_{\mathbf{n}}(I_{E(\mathbf{n})} \otimes \theta)\wt{T}_{\mathbf{n}}^*h_n \hspace{1cm} \mbox{for}\: \bigoplus_{\mathbf{n} \in \mathbb{N}^k_0}h_{\mathbf{n}} \in \mathcal{H}.$$
\end{notation}
\begin{definition}
	An operator $\theta: \mathcal{W}_{\mathcal{H}} \to \mathcal{K}$ will be called
\begin{enumerate}
\item {\rm inner} if $M_{\theta}$ is an isometry,
\item {\rm outer} if $\overline{M_{\theta}\mathcal{H}}=\mathcal{K}$.
\end{enumerate}\end{definition}

\begin{proposition}
Let $\theta: \mathcal{W}_{\mathcal{H}} \to \mathcal{K}$ be an operator such that $M_{\theta} $ is a contraction.
\begin{enumerate}
\item $\theta$ is inner if and only if $\theta$ is an isometry and $\theta \mathcal{W}_{\mathcal{H}} $ is a wandering subspace for $(\psi, V^{(1)}, \dots, V^{(k)}).$
\item $\theta$ is outer if and only if    $\theta \mathcal{W}_{\mathcal{H}} $ is cyclic for  $(\psi, V^{(1)}, \dots, V^{(k)})$, i.e. $$\bigvee_{\mathbf{n} \in \mathbb{N}^k_0}\wt{V}_{\mathbf{n}}(E(\mathbf{n}) \otimes  \theta \mathcal{W}_{\mathcal{H}} )=\mathcal{K}.$$
    \item $\theta$ is inner and  outer if and only if $\theta$ is a unitary operator from $ \mathcal{W}_{\mathcal{H}}$ to $\mathcal{W}_{\mathcal{K}}.$
\end{enumerate}
\end{proposition}
\begin{proof}
(1) Assume $M_{\theta}$ is an isometry. Let $h_1, h_2 \in\mathcal{W}_{\mathcal{H}},$ and $ \xi_{\mathbf{n}} \in E(\mathbf{n}), \mathbf{n} \in \mathbb{N}^k_0 \setminus \{0\} ,$ we have
\begin{align*}
\langle \theta h_1, \wt{V}_\mathbf{n}(I_{E(\mathbf{n})} \otimes \theta)(\xi_{\mathbf{n}} \otimes h_2) \rangle  = & \langle \theta h_1, M_{\theta}\wt{T}_{\mathbf{n}}(\xi_{\mathbf{n}} \otimes h_2) \rangle \\
=& \langle h_1, \wt{T}_{\mathbf{n}}(\xi_{\mathbf{n}} \otimes h_2 \rangle)=0,
\end{align*}
since $\mathcal{W}_{\mathcal{H}}$ is wandering subspace for $(\sigma, T^{(1)}, \dots, T^{(k)}).$ Therefore, $\theta \mathcal{W}_{\mathcal{H}}$ is wandering subspace for $(\psi, V^{(1)}, \dots, V^{(k)}).$ \newline Conversely, let $h, k \in \mathcal{H}\left(=\bigoplus_{\mathbf{n} \in \mathbb{N}_0^k}\wt{T}_{\mathbf{n}}(E(\mathbf{n}) \otimes \mathcal{W}_{\mathcal{H}} ) \right),$ where $$ h= \sum_{\mathbf{n} \in \mathbb{N}^k_0}\wt{T}_{\mathbf{n}}(\xi_{\mathbf{n}} \otimes h_{\mathbf{n}}), k=\sum_{\mathbf{m} \in \mathbb{N}^k_0}\wt{V}_{\mathbf{m}}(\eta_{\mathbf{m}} \otimes w_{\mathbf{n}})$$ for some $ \xi_{\mathbf{n}}\in E(\mathbf{n}),\eta_{\mathbf{m}} \in E(\mathbf{m})$ and $ h_{\mathbf{n}},k_{\mathbf{n}} \in \mathcal{W}_{\mathcal{H}}.$ Since $\theta$ is an isometry and $\theta \mathcal{W}_{\mathcal{H}}$ is wandering subspace
\begin{align*}
\langle M_{\theta}h, M_{\theta}k \rangle =& \sum_{\mathbf{n}, \mathbf{m} \in \mathbb{N}^k_0}\langle M_{\theta}\wt{T}_{\mathbf{n}}(\xi_{\mathbf{n}} \otimes h_{\mathbf{n}}), M_{\theta}\wt{T}_{\mathbf{m}}(\eta_{\mathbf{m}} \otimes k_{\mathbf{m}})\rangle \\
=& \sum_{\mathbf{n}, \mathbf{m} \in \mathbb{N}^k_0}\langle \wt{V}_{\mathbf{n}}(\xi_{\mathbf{n}} \otimes \theta h_{\mathbf{n}}), \wt{V}_{\mathbf{m}}(\eta_{\mathbf{m}} \otimes \theta k_{\mathbf{m}})\rangle \\
=& \sum_{\mathbf{n} \in \mathbb{N}^k_0}\langle \xi_{\mathbf{n}} \otimes \theta h_{\mathbf{n}}, \eta_{\mathbf{n}} \otimes \theta k_{\mathbf{n}}\rangle \\
=& \sum_{\mathbf{n} \in \mathbb{N}^k_0}\langle  \xi_{\mathbf{n}} \otimes  h_{\mathbf{n}}, \eta_{\mathbf{n}} \otimes k_{\mathbf{n}}\rangle= \langle h, k\rangle.
\end{align*}
Hence $M_{\theta}$ is an isometry.

The statement (2) is easy to prove and
(3) follows from Theorem \ref{MT3}.
\end{proof}
The following theorem for a covariant representation of a $C^*$-correspondence is a generalization of Popescu's version of Beurling-Lax theorem \cite[Theorem 2.2]{PoB}.
\begin{theorem}\label{MT4}
Let $(\sigma, T)$ be a analytic, isometric covariant representation of $E$ on the Hilbert space $\mathcal{H}$ and let $\mathcal{K}$ be closed subspace of $\mathcal{H}.$ Then  $\mathcal{K}$ is $(\sigma, T)$-invariant if and only if there exists a Hilbert space $\mathcal{W},$ a representation $\sigma_1$ of $\mathcal{M}$ on $\mathcal{W}$ and inner operator $\theta: \mathcal{W}_{\mathcal{F}(E) \otimes \mathcal{W}} \to \mathcal{H}$ such that $$\mathcal{K}= M_{\theta}(\mathcal{F}(E) \otimes \mathcal{W}),$$ where  $\mathcal{W}_{\mathcal{F}(E) \otimes \mathcal{W}}$ is  generating wandering subspace for the induced representation $(\rho, V)$  induced by $\sigma_1.$ In particular, if $\mathcal{K}=\mathcal{H},$ then $\theta$ is outer.
\end{theorem}
\begin{proof}
Let $\mathcal{K}$ be a $(\sigma, T)$-invariant subspace. It was shown in \cite[Theorem 2.4]{HS19}, there exists a wandering subspace $\mathcal{W}_{\mathcal{K}}$ for $(\sigma, T)|_{\mathcal{K}}$ such that $$\mathcal{K}=\bigoplus_{n \in \mathbb{N}_0}\mathfrak{L}_n(\mathcal{W}_{\mathcal{K}}).$$ Since $\mathcal{W}_{\mathcal{K}}$ is $\sigma(\mathcal{M})$-invariant, consider the induced representations $(\rho, V)$ on $\mathcal{H}_1:=\mathcal{F}(E) \otimes \mathcal{W}_{\mathcal{K}}$ induced by $\sigma|_{\mathcal{W}_{\mathcal{K}}}.$ Define the operator $A: \mathcal{F}(E)\otimes \mathcal{W}_{\mathcal{K}} \to \mathcal{H}$ by
\begin{align}\label{BU1}
A(\bigoplus_{n \in \mathbb{N}_0}\xi_n \otimes w_n)=\sum_{n \in \mathbb{N}_0}\wt{T}_n(\xi_n \otimes w_n)\:\:\:\:\:\mbox{for}\:\: \xi_n \in E^{\otimes n}\: \mbox{and}\: w_n \in \mathcal{W}_{\mathcal{K}},
\end{align}
then $A$ satisfies $AV(\xi)=T(\xi)A, A\rho(a)=\sigma(a)A$ for $ \xi \in E, a \in \mathcal{M} ,$ that is, $A$ is multi-analytic operator. It is easy to see that $A=M_{\theta}$ and $M_{\theta}$ is an isometry, where $\theta: \mathcal{W}_{\mathcal{H}_1} \to \mathcal{H}, \mathcal{W}_{\mathcal{H}_1}$ is generating wandering subspace for $(\rho, V).$ Finally, it follows from the Equation (\ref{BU1}) that $M_{\theta}\mathcal{H}_1=\mathcal{K}.$ The converse part is immediate.
\end{proof}

Let us recall a setup due to Popescu in \cite{PoB}:  Let $\Lambda$ be either  the set $\{1, 2, \dots, m\}$ where $m \in \mathbb{N},$ or $\mathbb{N}.$ For every  $k \in \mathbb{N},$  let $F(k, \Lambda) $ be the set of all functions from the set $\{1, 2, \dots, k\}$ to $\Lambda.$ Define $\mathfrak{F}:=\bigcup_{k=0}^{\infty}F(k, \Lambda),$ where $F(0, \Lambda)=\{0\}.$ A closed subspace $\mathcal{W} \subseteq \mathcal{H}$ is said to be generating wandering subspace for the sequence of isometries $\{T_n\}_{n=1}^m$ on the Hilbert space $\mathcal{H}$  if  for any distinct element $f,g \in \mathfrak{F}$ we have $T_f \mathcal{W}\bot T_g\mathcal{W},$ where $T_f$ stands for the product $T_{f(1)}T_{f(2)} \cdots T_{f(k)}$ and $T_0=I_{\mathcal{H}}, f \in \mathfrak{F},$ and $$\mathcal{H}=\bigoplus_{f \in \mathfrak{F}}T_f\mathcal{W}.$$ \\
 Let $\mathcal{H}$ be a Hilbert space. Define the Hilbert space of all formal power series with noncommuting indeterminates $X_{\lambda}(\lambda \in \Lambda)$ by \begin{align}\label{INP}
 \mathfrak{L}^2(\mathfrak{F}, \mathcal{H})=\big\{ \sum_{f \in \mathfrak{F}}h_{f}X_f \: : \: \: \: h_f \in \mathcal{H},\: \sum_{f \in \mathfrak{F}} \|h_{f}\|^2 < \infty\big\},
 \end{align}
  with the inner product is $$ \langle \sum_{f \in \mathfrak{F}}h_{f}X_f,\sum_{f \in \mathfrak{F}}g_{f}X_f\rangle =\sum_{f \in \mathfrak{F}}\langle h_f, g_f\rangle,$$
  where $X_f=X_{f(1)}X_{f(2)}\cdots X_{f(k)}$ for any $f \in F(k, \Lambda).$ Define the $\Lambda$-orthogonal shift $\mathfrak{S}=\{S_{\lambda}\}_{\lambda \in \Lambda}$ on $\mathfrak{L}^2(\mathfrak{F}, \mathcal{H})$ by
 \begin{align}\label{INP1}
 S_{\lambda}\big(\sum_{f \in \mathfrak{F}}h_{f}X_f\big)=\sum_{f \in \mathfrak{F}}h_{f}X_{\lambda}X_f.
 \end{align}
 It is easy to see that, the noncommuting operators $S_{\lambda}(\lambda \in \Lambda)$ is an isometry, $\mathcal{H}$ is generating wandering subspace for  the $\Lambda$-orthogonal shift $\mathfrak{S}$ and $\sum_{\lambda \in \Lambda }S_{\lambda}S_{\lambda}^* \leq I.$ As a consequence of Theorem \ref{MT4} we obtain the following corollary (which is \cite[Theorem 2.2]{PoB}).
   \begin{corollary}
Let $\{T_n\}$ be a finite (respectively, infinite )sequence of noncommuting isometries on the Hilbert space $\mathcal{H}$ such that $\sum T_nT^*_n \leq I$ and $\{T_n\}$ has generating wandering subspace property. Let $\mathcal{K}$ be a closed subspace of $\mathcal{H}.$ Then $\mathcal{K}$ is joint $\{T_n\}$-invariant if and only if there exists a Hilbert space $\mathcal{H}_1,$ a finite (respectively, infinite) sequence of orthogonal shift $\{S_{\lambda}\}$ on  $\mathfrak{L}^2(\mathfrak{F}, \mathcal{H}_1)$  and inner operator $\theta: \mathcal{H}_1 \to \mathcal{H}$ such that $\mathcal{K}=M_{\theta}\mathfrak{L}^2(\mathfrak{F}, \mathcal{H}).$
 \end{corollary}
In the following proposition we characterize the reducing subspace of  the covariant representations $(\sigma, T^{(1)}, \dots, T^{(k)})$ of $\mathbb{E}$ on the Hilbert space $\mathcal{H}.$

\begin{proposition}\label{BU2}
Let $(\sigma, T^{(1)}, \dots, T^{(k)})$  be a doubly commuting isometric  covariant representations of $\mathbb{E}$ on the Hilbert spaces $\mathcal{H}$  such that each $(\sigma, T^{(i)})$ is analytic. Then $\mathcal{K}$ is a $(\sigma, T^{(1)}, \dots, T^{(k)})$- reducing  subspace of $\mathcal{H}$ if and only if
$$\mathcal{K}=\bigoplus_{\mathbf{n} \in \mathbb{N}_0^k}\wt{T}_{\mathbf{n}}(E(\mathbf{n}) \otimes \mathcal{W}),$$ where $\mathcal{W}= \bigcap_{i=1}^k \mbox{ker}\wt{T}^{(i)^*}|_{\mathcal{K}}$
\end{proposition}
\begin{proof}
The necessary part follows from Corollary \ref{main1}.  Conversely, it is obvious that $\mathcal{K}$ is $(\sigma, T^{(1)}, \dots, T^{(k)})$-invariant. Let $i \in I_k$ be fixed and $\mathbf{n}=(n_1, \cdots, n_k) \in \mathbb{N}^k_0.$ If $n_i=0,$ then  $$\widetilde{T}^{(i)*}\widetilde{T}_{\mathbf{n}}=(I_{E_i} \otimes \widetilde{T}_{\mathbf{n}})(t_{\mathbf{n},i} \otimes I_{\mathcal{H}})(I_{\mathbb{E}(\mathbf{n})} \otimes \widetilde{T}^{(i)*}),$$
where $t_{\mathbf{n},i}: \mathbb{E}(\mathbf{n}) \otimes E_i \to E_i \otimes \mathbb{E}(\mathbf{n})$ is the isomorphism which is a composition of the isomorphism $t_{i,j}: E_i \otimes E_j \to E_j \otimes E_i.$ Since $\mathcal{W}= \bigcap_{j=1}^k \mbox{ker}\wt{T}^{(i)^*}|_{\mathcal{K}},$
 $ \widetilde{T}^{(i)^*}\mathcal{W}=\{0\}.$ Thus, it follows from the above equation we get $\widetilde{T}^{(i)*}\widetilde{T}_{\mathbf{n}}(\mathbb{E}(\mathbf{n}) \otimes \mathcal{W}) =\{0\},$ for $n \in \mathbb{N}^k_0$ with $n_i=0.$
Suppose that $n_i>0, $  $\widetilde{T}^{(i)*}\widetilde{T}_{\mathbf{n}}(\mathbb{E}(\mathbf{n}) \otimes \mathcal{W})=\widetilde{T}_{(n_i-1)}^{(i)}(I_{E_i^{\otimes (n_i-1)}} \otimes \widetilde{T}^{(i)}_{\mathbf{n}-n_i})(\mathbb{E}(\mathbf{n}) \otimes \mathcal{W}) \subset \mathcal{K}$ because $\mathcal{K}$ is $(\sigma, T^{(1)}, \dots, T^{(k)})$-invariant. Therefore $\widetilde{T}^{(i)*}\widetilde{T}_{\mathbf{n}}(E(\mathbf{n}) \otimes \mathcal{W}) \subset \mathcal{K}$ for all $\mathbf{n} \in \mathbb{N}^k_0$ and hence we get the desired result.
\end{proof}

\begin{remark}
Let $\pi$ be the representation of $\mathcal{M}$ on the Hilbert space $\mathcal{H}_2$ and let $(\rho, S^{(1)}, \dots , S^{(k)})$ be the  induced  representation on $\mathcal{H}_1:=\mathcal{F}(\mathbb{E}) \otimes \mathcal{H}_2$ induced by $\pi.$ Let $\mathcal{K}$ be the subspace of $\mathcal{H}_1$ such that $$\mathcal{K}=\bigoplus_{\mathbf{n} \in \mathbb{N}_0^k}\wt{S}_{\mathbf{n}}(\mathbb{E}(\mathbf{n}) \otimes \mathcal{W}),$$ for some closed subspace $\mathcal{W}$ of $\mathcal{H}_1.$ For $1 \leq i \leq k,$ we obtain
$$\mbox{ker}\wt{S}^{(i)^*}|_{\mathcal{K}}= \bigoplus_{\mathbf{n}\in \mathbb{N}^k_0,\: n_i=0}\mathbb{E}(\mathbf{n}) \otimes \mathcal{W} .$$
Thus $\bigcap_{i=1}^k\mbox{ker}\wt{S}^{(i)^*}|_{\mathcal{K}}=\mathcal{W}.$ Therefore, the sufficient part of Proposition \ref{BU2}, $\mathcal{W}=\bigcap_{i=1}^k \mbox{ker}\wt{T}^{(i)^*}|_{\mathcal{K}}$ is not necessary to prove $\mathcal{K}$ is $(\sigma, T^{(1)}, \dots , T^{(k)})$-reducing for the case of  induced representation $(\rho, S^{(1)}, \dots , S^{(k)}).$ \\
From Proposition \ref{BU2}, $\mathcal{W}$ is $\sigma$-invariant.  Let $(\rho, S^{(1)}, \dots , S^{(k)})$ be a induced representations on $\mathcal{H}_1:=\mathcal{F}(\mathbb{E}) \otimes \mathcal{W} $ induced by $\sigma|_{\mathcal{W}}.$ We define  the operator $A: \mathcal{F}(\mathbb{E}) \otimes \mathcal{W} \to \mathcal{K}$ by
\begin{align}\label{BU3}
A\left(\bigoplus_{\mathbf{n} \in \mathbb{N}_0^k} \xi_{\mathbf{n}} \otimes w_{\mathbf{n}}\right)=\sum_{\mathbf{n} \in \mathbb{N}_0^k}\wt{T}_{\mathbf{n}}(\xi_{\mathbf{n}} \otimes w_{\mathbf{n}}),
\end{align}
where $\xi_{\mathbf{n}} \in E(\mathbf{n}), w_{\mathbf{n}} \in \mathcal{W}.$  Then it is easy to see that $A$ is multi-Analytic  and hence $A=M_{\theta},$ where $\theta:=A|_{\mathcal{W}_{\mathcal{H}_1}}: \mathcal{W}_{\mathcal{H}_1} \to \mathcal{H}, \: \mathcal{W}_{\mathcal{H}_1}$ is generating wandering subspace for $(\rho, S^{(1)}, \dots , S^{(k)}),$ in fact $\mathcal{W}_{\mathcal{H}_1}=\mathcal{W}.$  From Equation (\ref{BU3}), $M_{\theta}$ is an isometry and $$\mathcal{K}=M_{\theta}\mathcal{H}_1.$$
\end{remark}
Next we explore the invariant subspaces for the isometric covariant representations $(\sigma, T^{(1)},$ $\dots,T^{(k)}),$ when $ k \geq 2.$

\begin{definition}
Let $(\sigma, T^{(1)}, \dots, T^{(k)})$  be a completely bounded covariant representation of $\mathbb{E}$ on the Hilbert space $\mathcal{H}.$ An  $(\sigma, T^{(1)}, \dots, T^{(k)})$- invariant subspace $\mathcal{K}$ is said to doubly commuting subspace if the covariant representation $(\sigma, T^{(1)}, \dots, T^{(k)})|_{\mathcal{K}}$ is  doubly commuting of $\mathbb{E}$ on the Hilbert space $\mathcal{K},$ that is,
$${\wt{T}|_{\mathcal{K}}^{(i)}}^{*}{\wt{T}|_{\mathcal{K}}^{(j)}}=(I_{E_i} \otimes {\wt{T}|_{\mathcal{K}}^{(j)}})(t_{j,i} \otimes I_{\mathcal{K}})(I_{E_i} \otimes{\wt{T}|_{\mathcal{K}}^{(i)}}^{*})$$ where $i,j \in \{1,2, \dots ,k\}$ with $ i \neq j.$
\end{definition}

Now we present a wandering subspace theorem concerning doubly commuting subspaces of $\mathcal{H}.$
\begin{theorem}\label{DCS1}
Let $(\sigma, T^{(1)}, \dots, T^{(k)})$ be a doubly commuting completely bounded, covariant representation of the product system $\mathbb{E}$ on a Hilbert space $\mathcal{H}$ such that  $(\sigma, T^{(1)}, \dots, $ $T^{(k)})$ is analytic  and satisfies one of the following properties:
\begin{enumerate}
\item[(1)] $(\sigma, T^{(i)})$ is concave for each $i=1, \dots ,k,$
\item[(2)]   $\|(I_{E_i} \otimes \wt{T}^{(i)})(\zeta)+\xi\|^2 \leq 2(\|\zeta\|^2+\|\wt{T}^{(i)}(\xi)\|^2),$ for each $\zeta \in E_i^{\otimes 2} \otimes \mathcal{H},\xi \in E_i \otimes \mathcal{H}, i= 1, \dots ,k.$
    \item[(3)]   $\|\wt{T}^{(i)}(\xi_i)\|^2+\|\wt{T}^{(i)*}_2\wt{T}^{(i)}(\xi_i)\|^2 \leq 2\|\wt{T}^{(i)*}\wt{T}^{(i)}(\xi_i )\|^2, \:  \xi_i \in E_i \otimes \mathcal{H}.$

\end{enumerate}
Let $\mathcal{K}$ be a $(\sigma, T^{(1)}, \dots, T^{(k)})$- doubly commuting subspace. Then  $$\mathcal{K}=\bigvee_{\mathbf{n} \in \mathbb{N}^k_0}T_{\mathbf{n}}(E(\mathbf{n}) \otimes \mathcal{W}'_{\mathcal{K}})$$ for some wandering subspace $\mathcal{W}'_{\mathcal{K}}$ for $(\sigma, T^{(1)}, \dots, $ $T^{(k)})|_{\mathcal{K}}.$
\end{theorem}
\begin{proof}
Since $\mathcal{K}$ is $(\sigma, T^{(i)})$-invariant and $(\sigma, T^{(i)})$ is analytic for each $1 \leq i \leq k,$ then $(\sigma, T^{(i)})|_{\mathcal{K}}$ is analytic. Also $(\sigma, T^{(1)}, \dots, T^{(k)})|_{\mathcal{K}}$ is doubly commuting covariant representation on $\mathbb{E}$, therefore the desired result follows from Theorem \ref{MT2}.
\end{proof}
\begin{remark}
Let $(\sigma, T^{(1)}, \dots, T^{(k)})$  be a doubly commuting isometric  covariant representations of $\mathbb{E}$ on the Hilbert space $\mathcal{H}$  such that each $(\sigma, T^{(i)})$  is analytic. From the above Theorem \ref{DCS1}, $$\mathcal{K}=\bigoplus_{\mathbf{n} \in \mathbb{N}_0^k}\wt{T}_{\mathbf{n}}(E(\mathbf{n}) \otimes \mathcal{W})$$ for some wandering subspace $\mathcal{W}.$ Moreover,  the corresponding generating wandering subspace is $\mathcal{W}=\bigcap_{i=1}^k(\mathcal{K} \ominus \wt{T}^{(i)}(E_i \otimes \mathcal{K})).$
\end{remark}
The following theorem gives necessary and sufficient for doubly commuting  and generating wandering subspace for isometric covariant representations.

\begin{theorem}\label{MT5}
Let $(\sigma, T^{(1)}, \dots, T^{(k)})$  be a doubly commuting isometric  covariant representations of  $\mathbb{E}$ on a Hilbert space $\mathcal{H}$  such that each $(\sigma, T^{(i)})$  is analytic. Let $\mathcal{K}$ be a $(\sigma, T^{(1)}, \dots, T^{(k)})$-invariant. Then  $\mathcal{K}$ is  $(\sigma, T^{(1)}, \dots, T^{(k)})$- doubly commuting subspace if and only if there exist a doubly commuting isometric  covariant representations  $(\rho, V^{(1)}, \dots, V^{(k)})$ of $\mathbb{E}$ on the Hilbert space $\mathcal{H}_1$, $(\sigma, V^{(i)})$  is analytic and an inner operator $\theta:\mathcal{W}_{\mathcal{H}_1} \to \mathcal{H}$ such that
$$\mathcal{K}=M_{\theta}\mathcal{H}_1.$$ In particular, if $\mathcal{K}=\mathcal{H},$ then $\theta$ is outer.
\end{theorem}
\begin{proof}
By Corollary \ref{MT3}  we have  $$\mathcal{K}=\bigoplus_{\mathbf{n} \in \mathbb{N}_0^k}\wt{T}_{\mathbf{n}}(E(\mathbf{n}) \otimes \mathcal{W}),$$
where $\mathcal{W}=\bigcap_{i=1}^k(\mathcal{K} \ominus \wt{T}^{(i)}(E_i \otimes \mathcal{K})).$ Since $\mathcal{W}$ is $\sigma(\mathcal{M})$-invariant, we define the representation $\rho$ of $\mathcal{M}$ on the Hilbert space $\mathcal{W}$ by  $\rho(a):=\sigma(a)|_{\mathcal{W}}, a \in \mathcal{M}.$ Consider the induced representation $(\rho,S^{(1)}, \dots ,S^{(k)})$  induced by $\rho,$ now we define the Hilbert space $\mathcal{H}_1$ by $\mathcal{H}_1=\mathcal{F}(\mathbb{E}) \otimes_{\rho} \mathcal{W}$ and the bounded operator $V:\mathcal{H}_1 \to \mathcal{H} $ by $$V\left(\sum_{\mathbf{n}\in \mathbb{N}^k_0}\xi_{\mathbf{n}} \otimes_{\rho} w_{\mathbf{n}} \right)=\sum_{\mathbf{n}\in \mathbb{N}^k_0}\wt{T}_{\mathbf{n}}(\xi_{\mathbf{n}} \otimes_{\rho} w_{\mathbf{n}}), \:\:\mbox{where}\: \sum_{\mathbf{n}\in \mathbb{N}^k_0}\xi_{\mathbf{n}} \otimes w_{\mathbf{n}} \in  \mathcal{F}(\mathbb{E}) \otimes_{\rho} \mathcal{W}.$$
Observe that
\begin{align*}
\sum_{\mathbf{n}\in \mathbb{N}^k_0}\|\xi_{\mathbf{n}} \otimes_{\rho} w_{\mathbf{n}}\|^2= &\sum_{\mathbf{n}\in \mathbb{N}^k_0} \langle w_{\mathbf{n}}, \rho(\langle \xi_{\mathbf{n}}, \xi_{\mathbf{n}} \rangle)w_{\mathbf{n}} \rangle=\sum_{\mathbf{n}\in \mathbb{N}^k_0} \langle w_{\mathbf{n}}, \sigma(\langle \xi_{\mathbf{n}}, \xi_{\mathbf{n}} \rangle)w_{\mathbf{n}} \rangle \\
=& \sum_{\mathbf{n}\in \mathbb{N}^k_0}\|\xi_{\mathbf{n}} \otimes w_{\mathbf{n}}\|^2=\sum_{\mathbf{n}\in \mathbb{N}^k_0}\|\wt{T}_{\mathbf{n}}(\xi_{\mathbf{n}} \otimes w_{\mathbf{n}})\|^2,
\end{align*}
 thus $V$ is an isometry. Moreover, for all $\xi_{\mathbf{n}} \in E(\mathbf{n}),\eta_{\mathbf{m}}\in E(\mathbf{m}), w \in \mathcal{W}, \mathbf{n},\mathbf{m} \in \mathbb{N}^k_0$ we have
 \begin{align*}
 VS_{\mathbf{n}}(\xi_{\mathbf{n}})(\eta_{\mathbf{m}} \otimes w)=&V(\xi_{\mathbf{n}}\otimes \eta_{\mathbf{m}} \otimes w)=\wt{T}_{\mathbf{n}+\mathbf{m}}(\xi_{\mathbf{n}}\otimes \eta_{\mathbf{m}} \otimes w) \\
 =&\wt{T}_{\mathbf{n}}(\xi_{\mathbf{n}} \otimes V(\eta_{\mathbf{m}} \otimes w))=T(\xi_{\mathbf{n}})V(\eta_{\mathbf{m}} \otimes w),
 \end{align*}
 that is, $VS_{\mathbf{n}}(\xi_{\mathbf{n}})=T_{\mathbf{n}}(\xi_{\mathbf{n}})V$ for all $\xi_{\mathbf{n}} \in E(\mathbf{n}), \:\mathbf{n} \in \mathbb{N}^k_0.$
Therefore $V=M_{\theta},$ $\theta: \mathcal{W}_{\mathcal{H}_1} \to \mathcal{H}$ is inner,  in fact $\mathcal{W}_{\mathcal{H}_1}=\mathcal{W}.$ it is easy to see that $$\mbox{ran}V=\mbox{ran}M_{\theta}=\mathcal{K},$$ that is, $M_{\theta}\mathcal{H}_1=\mathcal{K}.$
To prove the converse part, let  $\mathcal{K}=M_{\theta}\mathcal{H}_1$ be a $(\sigma, T^{(1)}, \dots, T^{(k)})$-invariant such that $\theta $ is inner.

 Then $$M_{\theta}M_{\theta}^*=P_{\mathcal{K}},$$ where $P_{\mathcal{K}}$ is orthogonal projection oh $\mathcal{H}$ onto $\mathcal{K}.$ Then for all $i \neq j,$ we have
 \begin{align*}
& (I_{E_i} \otimes \wt{T}|_{\mathcal{K}}^{(j)})(t_{j,i} \otimes I_{\mathcal{H}_1})(I_{E_j} \otimes  \wt{T}|_{\mathcal{K}}^{(i)^*})\\
 =&(I_{E_i} \otimes P_{\mathcal{K}})(I_{E_i} \otimes \wt{T}^{(j)})(t_{j,i} \otimes P_{\mathcal{K}})(I_{E_j} \otimes \wt{T}^{(i)^*})\\
 =&(I_{E_i} \otimes P_{\mathcal{K}})(I_{E_i} \otimes \wt{T}^{(j)}(I_{E_i} \otimes M_{\theta}))(t_{j,i} \otimes I_{\mathcal{H}_1})(I_{E_j} \otimes (\wt{T}^{(i)}(I_{E_i} \otimes M_{\theta}))^*)\\
 =& (I_{E_i} \otimes P_{\mathcal{K}})(I_{E_i} \otimes M_{\theta})(I_{E_i} \otimes \wt{V}^{(j)})(t_{j,i} \otimes I_{\mathcal{H}_1})(I_{E_j} \otimes \wt{V}^{(i)^*})(I_{E_j} \otimes M_{\theta}^*)\\
 =&(I_{E_i} \otimes P_{\mathcal{K}})(I_{E_i} \otimes M_{\theta})\wt{V}^{(i)^*}\wt{V}^{(j)}(I_{E_j} \otimes M_{\theta}^*)\\
 =&(I_{E_i} \otimes P_{\mathcal{K}})\wt{T}^{(i)^*}\wt{T}^{(j)}(I_{E_j} \otimes P_{\mathcal{K}})={\wt{T}|_{\mathcal{K}}^{(i)}}^{*}{\wt{T}|_{\mathcal{K}}^{(j)}},
 \end{align*}
 since $(I_{E_i} \otimes M_{\theta})\wt{V}^{(i)^*}=(I_{E_i} \otimes P_{\mathcal{K}})\wt{T}^{(i)^*}.$ Hence $(\sigma, T^{(1)}, \dots, T^{(k)})|_{\mathcal{K}}$ is doubly commuting covariant representations of $\mathbb{E}$ on $\mathcal{H}.$
 \end{proof}
%

 Let $\Lambda $ be a set $\{1, \dots, k\}(k \in \mathbb{N})$ and let $\mathcal{H}$ be a Hilbert space. Define the Hilbert space $\mathfrak{L}^2(\mathfrak{F}, \mathcal{H})$ as in Equation (\ref{INP}) with commuting indeterminates $X_{\lambda}(\lambda \in \Lambda).$ Similarly, define the $\Lambda$-orthogonal shift $\mathfrak{S}=\{S_{\lambda}\}_{\lambda \in \Lambda}$  as in Equation (\ref{INP1}) with commuting indeterminates $X_{\lambda}(\lambda \in \Lambda)$ on $\mathfrak{L}^2(\mathfrak{F}, \mathcal{H}).$
 It is easy to see that, the  operators $\{S_{\lambda}:1 \leq \lambda \leq k\}$ with commuting indeterminates  $X_{\lambda}$ is  doubly commuting isometry and  $\mathcal{H}$ is generating wandering subspace for  the $\Lambda$-orthogonal shift $\mathfrak{S}.$ This combined with the following corollary which characterize the doubly commuting subspace.

\begin{corollary}\label{PCr}
Let $T=(T_1, \dots, T_k)$ be an $q$-doubly commuting isometries on $\mathcal{H}$ such that  $T_i$ is analytic  for $1 \leq i \leq k.$ Let $\mathcal{K}$ be  an joint $T$-invariant subspace. Then $\mathcal{K}$ is doubly commuting subspace for $T$ if and only if  there exists a Hilbert subspace $\mathcal{W}$ of $\mathcal{H}$, a $\Lambda$-orthogonal shift $\{S_{\lambda}\}$ on $\mathfrak{L}^2(\mathfrak{F}, \mathcal{W})$ with commuting indeterminates $X_{\lambda}(\lambda \in \Lambda)$ and an inner operator $\theta: \mathcal{W} \to \mathcal{H}$ such that $\mathcal{K}=M_{\theta}\mathfrak{L}^2(\mathfrak{F}, \mathcal{W}).$
\end{corollary}
Let $\mathcal{H}_1$ and $\mathcal{H}_2$ be the Hilbert spaces and $B(\mathcal{H}_1, \mathcal{H}_2)$ be set of all bounded operator from $\mathcal{H}_1$ to $\mathcal{H}_2.$ Denote $H^{\infty}_{\mathcal{H}_1 \to \mathcal{H}_2}(\mathbb{D}^k)$ by  the set of all $B(\mathcal{H}_1, \mathcal{H}_2)$-valued bounded holomorphic functions on $\mathbb{D}^n$, that is,
$$H^{\infty}_{\mathcal{W}_1 \to \mathcal{W}}(\mathbb{D}^n)=\{f: \mathbb{D}^n \to B(\mathcal{H}_1, \mathcal{H}_2) \: \:\mbox{ is holomorphic }\: : \: \sup_{\mathbf{z} \in \mathbb{D}^k}\|f(\mathbf{z})\|_{B(\mathcal{H}_1, \mathcal{H}_2)} < \infty\}.$$ We call the operator valued function  $\Theta \in H^{\infty}_{\mathcal{W}_1 \to \mathcal{W}}(\mathbb{D}^n) $ is inner if $$\Theta(\mathbf{z})^*\Theta(\mathbf{z})=I_{\mathcal{H}_1}$$ for almost all $\mathbf{z} \in \mathbb{T}^k,$  where $\mathbb{T}$ is boundary of $\mathbb{D}.$
\begin{corollary}
Let $M_{\mathbf{z}}:=(M_{z_1}, \dots, M_{z_k})$ be a multiplication operators  on Hilbert  valued Hardy space $\mathcal{H}_{\mathcal{W}}^2(\mathbb{D}^k)$ and  $\mathcal{K}$ be a closed subspace of $\mathcal{H}_{\mathcal{W}}^2(\mathbb{D}^k).$ Then $\mathcal{K}$ is doubly commuting subspace  for  $M_{\mathbf{z}}$ if and only if there exists a closed subspace $\mathcal{W}_1$ of $\mathcal{W}$  and an inner function $\Theta \in H^{\infty}_{\mathcal{W}_1 \to \mathcal{W}}(\mathbb{D}^n)$ such that  $$\mathcal{K}=M_{\Theta}\mathcal{H}_{\mathcal{W}_1}^2(\mathbb{D}^k).$$
\end{corollary}
\begin{proof}
Since $M_{\mathbf{z}}$ is doubly commuting isometries on  $\mathcal{H}_{\mathcal{W}}^2(\mathbb{D}^k)$ and  $M_{z_i}$ is analytic for $1 \leq i \leq k.$ Then by Corollary \ref{PCr}, there exists a closed subspace $\mathcal{W}_1$ of $\mathcal{W}$, a  orthogonal shift $\{S_i: 1 \leq i \leq k\}$ on $\mathfrak{L}^2(\mathfrak{F}, \mathcal{W}_1)$ with commuting indeterminates $X_i(\lambda \in \Lambda)$ and an inner operator $\theta: \mathcal{W} \to \mathcal{H}$ such that $\mathcal{K}=M_{\theta}\mathfrak{L}^2(\mathfrak{F}, \mathcal{W}).$ Since $\{S_i: 1 \leq i \leq k\}$ is commuting operator, define the unitary operator $$U_1:\mathcal{H}_{\mathcal{W}_1}^2(\mathbb{D}^k) \to \mathfrak{L}^2(\mathfrak{F}, \mathcal{W}_1)$$ by $U_1(\sum_{\mathbf{n} \in \mathbb{N}_0^k}\mathbf{z}^nh_{\mathbf{n}})=\sum_{\mathbf{n} \in \mathbb{N}_0^k}S^{\mathbf{n}}h_{\mathbf{n}}$ for all $h_{\mathbf{n}} \in \mathcal{W}_1.$ It follows that $U_1M_{z_i}=S_iU_1$ for all $1 \leq i \leq k.$ Now, we define the isometry $V: \mathcal{H}_{\mathcal{W}_1}^2(\mathbb{D}^k) \to \mathcal{H}_{\mathcal{W}}^2(\mathbb{D}^k)$ by $V=M_{\theta}U_1.$ By the definition of multi-analytic operator $M_{\theta},$ we have  $VM_{z_i}=M_{\theta}S_iU_1=M_{z_i}M_{\theta}U_1=M_{z_i}V.$ Hence the isometry $V$ is an bimodule map and it follows from \cite{BLTT}, $V=M_{\Theta}$ for some inner function $\Theta \in \mathcal{H}_{\mathcal{W}_1}^2(\mathbb{D}^k).$ Also, $M_{\Theta}\mathcal{H}_{\mathcal{W}}^2(\mathbb{D}^k)=M_{\theta}U_1\mathcal{H}_{\mathcal{W}}^2(\mathbb{D}^k)=M_{\theta}\mathfrak{L}^2(\mathfrak{F}, \mathcal{W})=\mathcal{K}.$ Converse part follows from the proof of the converse of Theorem \ref{MT5}.
\end{proof}

\subsection*{Acknowledgment}
Shankar V. is grateful to The LNM Institute of Information Technology for providing research facility and warm hospitality during a visit in December 2019. Shankar V. is supported by CSIR Fellowship (File No: 09/115(0782)/2017-EMR-I).

\end{document}